\documentclass[11pt]{article}
\usepackage[colorlinks=true,pdfborder=001,linkcolor=blue,citecolor=blue]{hyperref}
\usepackage{url}
\usepackage{stackengine}
\usepackage{nccmath}
\usepackage{mathtools}

\usepackage{mathrsfs}
\usepackage[title]{appendix}
\usepackage{adjustbox}
\usepackage{subcaption}
\usepackage{pgfplots}
\usepackage{tikz}
\usepackage{tikz,fullpage}
\usepackage{rotating}
\usepackage{tabu}
\usepackage{hhline}
\usepackage{multirow}
\usepackage{float}
\restylefloat{table}
\usepackage{adjustbox}
\usepackage[round]{natbib}
\usepackage{booktabs}
\usepackage{array}
\usepackage{tabularx}
\usepackage[utf8]{inputenc}
\usepackage{amssymb,amsmath}
\usepackage{amsfonts}
\usepackage{latexsym}
\usepackage[english]{babel}
\usepackage[T1]{fontenc}
\usepackage{eepic}
\usepackage{epic}
\usepackage{epsfig}
\usepackage{graphics,color}
\usepackage{verbatim}
\usepackage{lscape}
\usepackage{amsthm}
\usepackage{amsmath}
\usepackage{placeins}
\usepackage{float}
\usepackage{caption}
\usepackage{color}
\usepackage{setspace}
\usepackage[ruled,linesnumbered]{algorithm2e}
\usepackage{frcursive}

\usepackage{mathrsfs}
\usepackage{soul}
\usepackage[normalem]{ulem}
\definecolor{rmvgray}{rgb}{0.55,0.55,0.55}

\definecolor{addblue}{rgb}{0.10,0.30,0.85}

\definecolor{addorange}{rgb}{0.95,0.45,0.0}
\usepackage{tikz}
\usepackage{tikz,fullpage}
\usetikzlibrary{arrows,petri,topaths,automata}
\usepackage[short,nocomma]{optidef}
\usetikzlibrary{shadows,positioning,shapes.geometric,arrows.meta,calc,fit,backgrounds}

\usepackage[colorinlistoftodos]{todonotes}

\makeatletter

\renewcommand*\thesection{\arabic{section}}

\newtheorem{prop}{Proposition}

\newtheorem{defi}{Definition}
\newtheorem{lemm}{Lemma}
\newtheorem{rema}{Remark}
\newtheorem{theo}{Theorem}
\newtheorem{exam}{Example}

\newtheorem{assum}{Assumption}
\newif\ifbold
\boldfalse
\boldtrue
\newcommand{\bbf}{\ifbold\bgroup\bf\fi}
\newcommand{\ebf}{\ifbold\egroup\fi}

\renewcommand{\textbf}[1]{\begingroup\bfseries\mathversion{bold}#1\endgroup}
\makeatletter    
\renewcommand{\section}{\@startsection {section}{1}{\z@}%
             {-2ex \@plus -1ex \@minus -.2ex}%
             {1ex \@plus.2ex}%
             {\normalfont\Large\rmfamily\bfseries}}
\renewcommand{\subsection}{\@startsection{subsection}{2}{\z@}%
             {-1.25ex\@plus -1ex \@minus -.2ex}%
             {.75ex \@plus .2ex}%
             {\normalfont\large\rmfamily\bfseries}}
\def\@listI{\leftmargin\leftmargini       
            \parsep .25ex \@plus .1ex     
            \topsep .25ex \@plus .1ex     
            \itemsep \parsep}
\let\@listi\@listI
\@listi
\makeatother
\makeatletter

\@addtoreset{equation}{section}
\makeatother
\makeatletter

\@addtoreset{table}{section}
\makeatother

\usepackage[margin=1in]{geometry}
\usepackage{graphicx}
\graphicspath{ {./Figures/} }
\definecolor{purple}{rgb}{0.4,0.2,1}

\usepackage{titling}

\title{
\LARGE\bf The Proxy Benders Decomposition
\vspace{1ex}
}

\author{\large Changkun Guan,$^{1}$ El Mehdi Er Raqabi,$^{1,2,*}$ Mathieu Tanneau,$^{1}$ Pascal Van Hentenryck$^{1}$\\
\footnotesize$^1$\emph{H. Milton Stewart School of Industrial and Systems Engineering, Georgia Institute of Technology, Atlanta, USA}\\
\footnotesize$^2$\textit{Department of Operations and Decision Systems, Université Laval, Quebec, Canada}\\
\footnotesize$^*$\emph{Corresponding Author: eraqabi6@gatech.edu}\\
}

\date{}

\begin{document}
\maketitle

\vspace{2cm}
\begin{abstract}
\vspace{0.5cm}

Benders decomposition is a fundamental framework for solving large-scale mixed-integer optimization problems with complicating variables that, when fixed, yield significantly easier subproblems. However, classical Benders decomposition repeatedly solves highly similar subproblems and often exhibits zigzagging behavior across iterations, leading to slow convergence in large-scale settings. Motivated by the repetitive structure and parametric nature of Benders subproblems, this paper introduces the proxy Benders decomposition (Proxy-BD), a new decomposition framework in which subproblem optimization is replaced by certified optimization proxies rather than repeated exact solves.

The proposed proxy follows a self-supervised \emph{predict--project--and--complete} mechanism that produces dual-feasible solutions for generating provably valid Benders cuts. The framework preserves the theoretical validity of the decomposition independently of prediction quality through a projection-and-completion certification layer. A formal characterization of proxy-induced cuts is established, and the framework naturally extends to modern decomposition schemes, including branch-and-Benders-cut algorithms.

Computational experiments on large-scale facility location and network design problems demonstrate that Proxy-BD substantially reduces subproblem computational effort while maintaining near-optimal solution quality. On large-scale uncapacitated facility location instances up to $2000\times2000$, Proxy-BD achieves median optimality gaps below 0.5\%, yields up to $161\times$ median speedups, and reduces the number of generated cuts by more than $240\times$ on the largest instances. The computational gains consistently increase with recourse complexity, indicating that proxy-based inference scales substantially more favorably than repeated exact subproblem optimization in large-scale decomposition settings. These results highlight the potential of certified learned inference within decomposition algorithms and position Proxy-BD as a promising direction for scalable, learning-enhanced, and real-time mixed-integer optimization under repeated decision-making settings.

\vspace{0.2cm}
\par
\vspace{0.2cm}

\end{abstract}

\setlength{\parindent}{1em}
\setlength{\parskip}{0.5em}
\doublespacing

\section{Introduction}\label{sec:intro}

Mixed-integer programs (MIPs) frequently exhibit a structure in which a subset of the decision variables makes the problem hard to solve, while the remaining variables define large continuous substructures. This separation naturally suggests the use of decomposition methods. Among them, Benders decomposition (BD) remains one of the most widely used approaches in large-scale settings \citep{bnnobrs1962partitioning}. Consider the mixed-integer optimization problem
\begin{align}
    \min_{y \in Y,\, x \in X} \;\; & f^\top y + q^\top x \label{eq:mip} \\
    \text{s.t.} \;\; 
        & A y + B x \ge b, \nonumber
\end{align}
where $y \in Y \subseteq \mathbb{Z}^{m}$ denotes the complicating variables, and $x \in X \subseteq \mathbb{R}^n_+$ denotes the continuous recourse variables. Problems of this form arise in a broad range of applications, where fixing $y$ simplifies the constraint system, enabling efficient optimization of the $x$-variables.

For any fixed $\bar{y}$, the second-stage recourse subproblem (SP) is
\begin{align}
    Q(\bar{y}) := 
    \min_{x \in X} \left\{ q^\top x : Bx \ge b - A\bar{y} \right\}, \label{eq:subproblem}
\end{align}

In the standard-form case $X=\mathbb{R}^n_+$, its dual is given by

\begin{equation}
  Q(\bar y)
  \;=\;
  \max_{\lambda \in \mathcal{D}}
  \left\{ \lambda^\top(b-A\bar y) \right\},
  \qquad
  \mathcal{D} := \left\{ \lambda \ge 0 \;:\; B^\top \lambda \le q \right\}.
  \label{eq:2stage-dual}
\end{equation}

Let $\mathcal{P}_{\mathcal{D}}$ and $\mathcal{R}_{\mathcal{D}}$ denote the sets of extreme points and extreme rays of the polyhedron $\mathcal{D}$, respectively. The Benders master problem introduces an auxiliary variable $\theta$ to represent the
second-stage cost $Q(y)$. If the primal subproblem \eqref{eq:subproblem} is feasible for $\bar y$, cuts are derived from dual-feasible multipliers: for any
$\lambda \in \mathcal{D}$, weak duality implies
\[
  Q(y) \;\ge\; \lambda^\top(b-Ay),
\]
and therefore the inequality
\begin{equation}
  \theta \;\ge\; \lambda^\top(b-Ay)
  \label{eq:opt-cut}
\end{equation}
is a valid Benders optimality cut. In classical BD, one typically
chooses $\lambda$ as an optimal dual solution for $\bar y$, which can be taken at an
extreme point $\bar\lambda \in \mathcal{P}_{\mathcal{D}}$. If the primal subproblem \eqref{eq:subproblem} is infeasible for $\bar y$, then, under the
standard assumption that $\mathcal{D}$ is nonempty, the dual problem \eqref{eq:2stage-dual}
is unbounded, and there exists an extreme ray $\bar r \in \mathcal{R}_{\mathcal{D}}$ such that
$\bar r^\top(b-A\bar y) > 0$. By Farkas' Lemma, any $y$ admitting feasible recourse must satisfy the feasibility cut
\begin{equation}
  0 \;\ge\; \bar r^\top(b-Ay).
  \label{eq:feas-cut}
\end{equation}

Substituting the value function $Q(y)$ into \eqref{eq:mip} yields the classical Benders master formulation

\begin{align}
    \min_{y,\theta} \quad & f^\top y + \theta \\
    \text{s.t.} \quad
    & \theta \ge \lambda^\top(b-Ay), \quad && \forall \lambda \in \mathcal{P}_{\mathcal{D}}, \\
    & 0 \ge r^\top(b-Ay), \quad && \forall r \in \mathcal{R}_{\mathcal{D}}, \\
    & y \in Y.
\end{align}

\noindent These constraints, known as Benders optimality and feasibility cuts, are not enumerated explicitly. Instead, they are generated iteratively: at each iteration, the master problem provides a trial solution $\bar{y}$ to the SP, which is solved to obtain either an optimality cut or a feasibility cut, which is then added to the restricted master problem (RMP). Since the RMP contains only a subset of the cuts, it represents a relaxation of the original problem and provides a lower bound (LB) on its optimal solution. Simultaneously, if the SP is feasible, it generates a feasible solution to the original problem, yielding an upper bound (UB). The Benders algorithm iterates until the difference between the UB and LB falls below a chosen threshold $\epsilon \geq 0$. Compared to the full MIP formulation, which includes all complicating variables, the SP in BD contains none of these variables, making it much simpler and faster to solve.

Due to its broad applicability, the BD has been successfully employed in a wide range of practical problems. These include production routing \citep{adulyasak2015benders}, electric vehicle management \citep{zhang2021values}, airline scheduling \citep{zeighami2019combining, cordeau2001benders}, water resources planning \citep{cai2001solving}, on-demand delivery systems \citep{liu2023demand}, hub location problems \citep{contreras2011benders}, locomotive assignment \citep{cordeau2001simultaneous}, the traveling salesman problem \citep{laporte1994priori}, vessel service planning \citep{wu2022vessel}, capacity expansion \citep{bloom1983solving}, budgeting applications \citep{keshvari2022budgeting}, and supply chain planning \citep{erraqabi2,erraqabi5}.

Despite its widespread success, the BD can be computationally intensive, exhibits zigzagging patterns, and often converges slowly. While these limitations may be negligible for small- or medium-sized instances, they become critical in large-scale applications. Consequently, significant research efforts have focused on accelerating the Benders convergence. These approaches can be divided into two categories. The first category aims to improve the lower bounds (LBs) provided by the RMP, while the second focuses on enhancing the upper bounds (UBs) obtained from the SP. In the first category, extensive work has been devoted to generating \emph{strong} or reinforced Benders cuts \citep{magnanti1981accelerating,codato2006combinatorial,fischetti2010note,bodur2017strengthened,fischetti2016benders,fischetti2017redesigning,rahmaniani2020benders}, which lead to tighter LBs. Additional acceleration strategies include incorporating valid inequalities, warm-starting, careful management of the branch-and-bound (B\&B) tree, and two-phase solution schemes that first generate cuts from a relaxed RMP before considering integer cuts. For UBs, there is currently no systematic approach to obtaining high-quality UBs from the SP. Existing methods rely primarily on problem-specific heuristics. For example, \cite{boland2016proximity} combine a proximity search with BD to obtain high-quality, albeit not necessarily optimal, solutions for various stochastic programming problems. There is, then, a practical need for approaches that can quickly generate good primal solutions. A comprehensive review by \citet{rahmaniani2017benders} provides an extensive overview of BD applications, challenges, and enhancement strategies.

However, despite decades of research on accelerating Benders decomposition, existing approaches still fundamentally rely on repeatedly solving exact subproblems throughout the decomposition process. In classical BD and its modern variants, each SP is solved independently, even though successive SPs often differ only in the right-hand side of \eqref{eq:subproblem}. As a result, a substantial portion of the computational effort is repeatedly spent solving highly similar optimization problems. This structural redundancy suggests that the response of the SP to variations in the master solution may be predictable. This observation motivates a different perspective on decomposition algorithms: rather than repeatedly solving nearly identical optimization problems, it becomes possible to amortize part of the subproblem computation through learned inference.

Recent work has explored the integration of machine learning into two-stage stochastic optimization and Benders-type decomposition algorithms. \cite{patel2022neur2sp} approximates the expected recourse function through a neural surrogate that replaces explicit second-stage optimization. While this substantially reduces online solution times, the resulting framework relies on learned value-function approximations whose prediction errors directly affect decision quality and are not accompanied by decomposition-level validity guarantees. \cite{larsen2024fast} use supervised learning to predict second-stage values and recourse information in continuous and integer L-shaped methods. Their approach replaces expensive recourse evaluations by oracle-trained predictors and accepts small optimality losses in exchange for computational gains. As a consequence, solution quality depends directly on prediction accuracy, and no mechanism enforces the validity of the approximated recourse information. More recently, \cite{li2025learning} employ reinforcement learning to adapt the inexactness level of generalized Benders decomposition throughout the optimization process. Learning therefore controls the decomposition algorithm, while subproblem optimization itself remains exact. Similarly, \cite{cai2026learning} learn cut-selection policies that reduce master-problem growth by deciding which Benders cuts to retain. Their framework improves cut management but leaves cut generation unchanged and continues to rely on exact separation. Collectively, these approaches demonstrate that learning can accelerate decomposition through recourse approximation, cut management, or algorithmic control. However, they do not replace the subproblem optimization process itself with a certified learning mechanism.

This paper introduces the \emph{proxy Benders decomposition} (Proxy-BD), a decomposition framework in which subproblem optimization is replaced by certified learned proxies rather than repeated exact solves. Instead of solving each SP from scratch, the proxy directly predicts dual information associated with the current master solution, thereby capturing the parametric response of the SP across iterations. The predicted dual information is then transformed into provably valid Benders cuts through a predict--project--and--complete certification mechanism. By exploiting the repetitive structure of the SPs, Proxy-BD amortizes the computational cost of repeated subproblem optimization while preserving the theoretical validity of the decomposition. The proposed framework is further motivated by the growing interest in optimization proxies for repeated and nested optimization tasks \citep{park2023self,tanneau2024dual}.

The contributions of this paper are fourfold:
\begin{itemize}
\item \textbf{Proxy-BD Framework.} Proxy-BD is introduced, a decomposition framework in which each subproblem solve is replaced by a predict--project--and--complete certification layer that produces provably valid Benders cuts through certified dual inference. By learning the parametric response of the subproblem to variations in the master solution, Proxy-BD amortizes repeated optimization effort. The framework unifies both optimality and feasibility cut generation within the same certification architecture.
\item \textbf{Certification Mechanism.} A projection-and-completion certification mechanism is established that preserves the theoretical validity of the decomposition independently of prediction quality. For feasibility cuts, the paper introduces a slice-normalized reformulation of Farkas certificates that converts the unbounded feasibility cone into a bounded learnable domain while preserving the extreme-ray structure of classical Benders feasibility cuts. A formal characterization of proxy-induced cuts is provided, and the resulting framework naturally extends to modern decomposition schemes, including branch-and-Benders-cut (B\&BC) algorithms, cut aggregation strategies, and accelerated Benders variants. The certification layer converts arbitrary proxy outputs into dual-feasible certificates, ensuring that learning affects cut strength rather than cut validity while maintaining compatibility with modern decomposition pipelines.
\item \textbf{Self-supervised Learning.} Unlike supervised approaches requiring optimal dual labels, Proxy-BD is trained through a self-supervised objective that maximizes certified dual bounds without requiring optimal dual solutions. This design avoids the ambiguity induced by dual degeneracy, where multiple dual-optimal certificates may generate identical Benders cuts, and aligns learning directly with cut strength rather than dual prediction accuracy. The framework leverages historical optimization information across iterations and instances, enabling knowledge reuse throughout the decomposition process and supporting repeated or real-time optimization settings under evolving data.
\item \textbf{Computational Results.} Computational experiments on large-scale facility location and network design problems demonstrate that Proxy-BD attains near-optimal solutions while achieving up to $161\times$ median speedups and reducing the number of generated cuts by more than $240\times$ against classical and accelerated Benders schemes. The computational gains consistently increase with recourse complexity, indicating it scales well to large instances where exact subproblem optimization becomes the dominant bottleneck. Moreover, the substantially smaller number of generated cuts suggests an implicit stabilization effect that may reduce the zigzagging behavior commonly observed in classical Benders decomposition.
\end{itemize}

The remainder of the paper is organized as follows. Section~\ref{sec:neural_optimality} establishes the proxy optimality cuts. Section~\ref{sec:neural_feasibility} handles subproblem infeasibility via slice-normalized Farkas certificates. Section~\ref{sec:proxy_benders} presents the proxy-BD framework. Section~\ref{section:experimental_design} and \ref{sec:computational_results} present the experimental design and the computational results, respectively. Section~\ref{section:conclusions} concludes with a discussion of implications and directions for future research.

\section{Proxy Optimality Cuts}
\label{sec:neural_optimality}

Consider the two-stage MIP \eqref{eq:mip} and its recourse value function $Q(\cdot)$ defined in \eqref{eq:subproblem}. This section first focuses on the complete-recourse setting, where the recourse problem is feasible and bounded for every $\bar y\in Y$, so that Benders decomposition generates only optimality cuts. In this setting, the optimization proxy is built around three components: (i) a \emph{predict--project--and--complete} mechanism that converts predicted multipliers into dual-feasible solutions at each RMP solution, (ii) structural completion procedures that are closed-form for many common recourse models, and (iii) a self-supervised training objective that maximizes certified dual bounds and backpropagates completion subgradients, eliminating the need for optimal dual labels required by standard supervised approaches.

\noindent \textbf{Notation.}
For a scalar $a\in\mathbb{R}$, define $(a)_+ := \max\{a,0\}$. For a vector $v$, $(v)_+$ is applied componentwise:
\(
(v)_+ := \bigl(\max\{v_i,0\}\bigr)_i.
\)
Equivalently, $(v)_+$ is the Euclidean projection of $v$ onto the nonnegative orthant.

\subsection{Predict--Project--and--Complete}
\label{sec:opt_block1}

By separating and re-organizing constraints, the MIP \eqref{eq:mip} can be rewritten as:

\begin{align}
    \min_{y \in Y,\, x \in X} \;\; & f^\top y + q^\top x \label{eq:reformulatemip} \\
    \text{s.t.} \;\; 
        & A y + B x \ge b, \nonumber\\
        & Gx \ge g, \nonumber
\end{align}
where $Gx\ge g$ are $y$-independent structure constraints. Define $d(y):=b-Ay$ and consider the recourse LP
\begin{align}
Q(\bar y)
:=
\min_{x\ge 0}\ \bigl\{ q^\top x:\ Bx \ge d(\bar y),\ \ Gx \ge g \bigr\}.
\label{eq:opt_sp_lp}
\end{align}

Its dual is

\begin{align}
\max_{\lambda,\mu}\quad & d(\bar y)^\top \lambda + g^\top \mu \label{eq:opt_sp_lp_dual}\\
\text{s.t.}\quad & B^\top \lambda + G^\top \mu \le q, \nonumber\\
& \lambda \ge 0,\ \mu \ge 0. \nonumber
\end{align}

Any dual-feasible pair $(\lambda,\mu)$ yields the valid Benders optimality cut
\begin{align}
\theta \;\ge\; d(y)^\top \lambda + g^\top \mu
\;=\; \alpha(\lambda,\mu) + \beta(\lambda)^\top y,
\label{eq:opt_cut_lp}
\end{align}
with coefficients $\alpha(\lambda,\mu):=b^\top\lambda+g^\top\mu$ and $\beta(\lambda):=-A^\top\lambda$.

The slope $\beta(\lambda)=-A^\top\lambda$ depends only on $\lambda$, while $\mu$ enters only the intercept $\alpha$. The linking-row multipliers therefore govern the cut geometry in the master space, while $\mu$ adjusts the intercept to the strongest value compatible with a fixed $\lambda$. The predict--project--complete pipeline below mirrors this asymmetry, letting the network learn $\lambda$ and recovering the optimal $\mu$ for any fixed $\lambda$ via linear-programming duality.

\noindent \textbf{Predict--and--Project.} Given a solution $\bar y$ from the RMP, the proxy predicts $\tilde\lambda\in\mathbb{R}^{n_\lambda}$, where $n_\lambda$ denotes the number of predicted constraints associated with $Bx\ge d(\bar y)$. The predicted multiplier is then projected onto the nonnegative orthant: $\hat\lambda:=(\tilde\lambda)_+\in\mathbb{R}^{n_\lambda}_+$.

\noindent \textbf{Complete.} Since $\hat\lambda\ge 0$ alone does not ensure dual feasibility, $\hat\lambda$ is completed by solving the strongest-completion LP in $\mu$ as follows. Given a projected multiplier $\hat\lambda\ge 0$, the remaining dual multiplier $\mu$ is recovered by solving
\begin{align}
\hat\mu(\hat\lambda)
\in
\arg\max_{\mu\ge 0}\ \Bigl\{ g^\top\mu:\ G^\top\mu \le q - B^\top\hat\lambda \Bigr\}.
\label{eq:opt_completion_lp}
\end{align}
Any feasible $\hat\mu(\hat\lambda)$ makes $(\hat\lambda,\hat\mu)$ dual-feasible for \eqref{eq:opt_sp_lp_dual}. The next proposition shows that choosing an optimizer yields the strongest bound among all dual-feasible completions of the fixed $\hat\lambda$.

\begin{assum}
\label{ass:opt_regular_completion_lp} 
For every projected multiplier $\hat\lambda\ge 0$, the completion LP \eqref{eq:opt_completion_lp} is feasible and bounded.
\end{assum}

\begin{theo}
\label{thm:opt_completion_existence_lp} Under Assumption~\ref{ass:opt_regular_completion_lp}, for any predicted output $\tilde\lambda$, there exists a completion $\hat\mu(\hat\lambda)$ solving \eqref{eq:opt_completion_lp} such that $(\hat\lambda,\hat\mu(\hat\lambda))$ is dual-feasible for \eqref{eq:opt_sp_lp_dual}. Consequently, the cut
$\theta \ge d(y)^\top\hat\lambda + g^\top\hat\mu(\hat\lambda)$
is valid for all $y$.
\end{theo}

\begin{proof}
Fix any predicted output $\tilde\lambda\in\mathbb{R}^{n_\lambda}$ and let $\hat\lambda=(\tilde\lambda)_+\ge 0$.
By Assumption~\ref{ass:opt_regular_completion_lp}, the completion LP \eqref{eq:opt_completion_lp} is feasible and bounded, and therefore admits an optimal solution. Let $\hat\mu(\hat\lambda)$ be any optimizer.
Then $\hat\mu(\hat\lambda)\ge 0$ and satisfies $G^\top\hat\mu(\hat\lambda)\le q-B^\top\hat\lambda$, hence
\(
B^\top\hat\lambda+G^\top\hat\mu(\hat\lambda)\le q.
\)
Therefore $(\hat\lambda,\hat\mu(\hat\lambda))$ is dual-feasible for \eqref{eq:opt_sp_lp_dual}.
Finally, weak duality between the primal recourse LP \eqref{eq:opt_sp_lp} and its dual \eqref{eq:opt_sp_lp_dual} implies
$d(y)^\top\hat\lambda+g^\top\hat\mu(\hat\lambda)\le Q(y)$ for all $y$, which proves that the cut
$\theta \ge d(y)^\top\hat\lambda+g^\top\hat\mu(\hat\lambda)$ is valid.
\end{proof}

\begin{theo}
\label{thm:opt_strongest_completion_lp}
Fix $\bar y$ and a projected multiplier $\hat\lambda\ge 0$. Any optimizer $\mu^\star(\hat\lambda)$ of \eqref{eq:opt_completion_lp} maximizes $g^\top\mu$ among all $\mu\ge 0$ satisfying $G^\top\mu \le q - B^\top\hat\lambda$, and therefore maximizes the cut right-hand side $d(\bar y)^\top\hat\lambda + g^\top\mu$ among all dual-feasible completions of $\hat\lambda$.
\end{theo}

\begin{proof}
Fix $\hat\lambda\ge 0$ and define the completion feasible region
\(
\mathcal{M}(\hat\lambda):=\{\mu\ge 0:\ G^\top\mu \le q - B^\top\hat\lambda\}.
\)
Problem \eqref{eq:opt_completion_lp} is exactly $\max\{g^\top\mu:\ \mu\in\mathcal{M}(\hat\lambda)\}$, so any optimizer maximizes $g^\top\mu$ over all dual-feasible completions.
Since $d(\bar y)^\top\hat\lambda$ is constant in $\mu$, maximizing $g^\top\mu$ also maximizes the cut right-hand side $d(\bar y)^\top\hat\lambda+g^\top\mu$ among all completions of the same $\hat\lambda$.
\end{proof}

Assumption~\ref{ass:opt_regular_completion_lp} makes the completion map well-defined for all projected multipliers. Theorems~\ref{thm:opt_completion_existence_lp} and \ref{thm:opt_strongest_completion_lp} formalize that predict--project--and--complete always yields a valid Benders optimality cut that is strongest among all dual-feasible completions of a fixed projected multiplier.

\begin{rema}[Choice of the predicted block]
\label{rem:predicted_block}
Partitioning the system into predicted-and-projected rows ($Bx \ge d(\bar{y})$) and completed rows ($Gx \ge g$) is a flexible modeling choice rather than a rigid restriction. Generally, the recourse dual can be split into any two blocks where the network predicts one and the completion problem recovers the other. Furthermore, either block can feature a $y$-dependent right-hand side. When completed rows depend on $y$, their corresponding dual multipliers contribute directly to the cut slope $\beta$, rather than just the constant $\alpha$. 

This specific configuration is illustrated in the facility-location derivation in Appendix~\ref{B}, where the network predicts the demand multipliers, and the completion step recovers the $y$-dependent capacity and bound multipliers that shape $\beta$. Ultimately, while the choice of partition heavily influences computational tractability and completion efficiency, it never compromises the validity of the resulting Benders cuts, which remains guaranteed by weak duality and Theorem~\ref{thm:opt_completion_existence_lp}.
\end{rema}

\subsection{Closed-form Completion}
\label{sec:opt_block2}

The completion step \eqref{eq:opt_completion_lp} generally requires solving an LP. However, many recourse models arising in practice involve bounded resources, such as capacity-constrained flows, inventories with storage limits, production with rate limits, and assignment fractions in $[0,1]$. In these settings, the completion map often admits a closed-form solution, similarly to the DLL-style completion framework of \citet{tanneau2024dual}. Consider the bounded recourse problem
\begin{align}
Q(\bar y)
:=
\min_{x}\ \bigl\{ q^\top x:\ Bx \ge d(\bar y),\ \ l \le x \le u \bigr\},
\label{eq:opt_sp_box}
\end{align}
where $l<u$ are finite componentwise bounds. A dual form introduces bound multipliers $z^l,z^u\ge 0$ with coupling constraint $B^\top\lambda + z^l - z^u = q$. For any projected $\lambda\ge 0$, the optimal completion is closed-form:
\begin{align}
z^l = (q - B^\top\lambda)_+,\qquad z^u = (B^\top\lambda - q)_+.
\label{eq:opt_box_completion_closed_form}
\end{align}
Thus, in bounded settings, enforcing feasibility reduces to elementwise operations. The closed-form completion is illustrated in Example~\ref{sec:opt_cflp_example}. 
\begin{rema}
The choice of which dual multipliers to predict can be guided by the property that the remaining dual variables admit a closed-form \emph{completion}. In this way, the proxy can efficiently produce the strongest feasible completion associated with the predicted multiplier without solving an additional LP. This closed-form property is central to computational efficiency and enables end-to-end gradient-based self-supervised training through the completion map.
\end{rema}

\begin{exam}\label{sec:opt_cflp_example} Consider the capacitated facility location problem (CFLP) in \eqref{eq:cflp_primal}. Using the scaling $x_{ij}\in[0,1]$ for the fraction of client demand $d_i$ served by
facility $j$, CFLP can be written compactly as
\begin{equation}
\label{eq:cflp_primal}
  \min\left\{
    \sum_{j} f_j y_j + \sum_{i,j} d_i c_{ij} x_{ij}
    \;:\;
    \begin{aligned}
      & \sum_{j} x_{ij} \ge 1 && \forall i,\\
      & \sum_{i} d_i x_{ij} \le s_j y_j && \forall j,\\
      & 0 \le x_{ij} \le y_j && \forall i,j,\\
      & y \in Y,
    \end{aligned}
  \right\}.
\end{equation}
where
\(
Y := \{y\in\{0,1\}^n:\ \sum_{j=1}^n s_j y_j \ge \sum_{i=1}^m d_i\},
\).
In a classical Benders master for CFLP, let $\theta$ be a scalar that represents the second-stage assignment cost and solve
\(
\min\{\sum_j f_j y_j + \theta:\ y\in Y,\ \theta\in\mathbb{R},\ \text{Benders cuts}\}.
\)
Benders uses optimality cuts of the form
\begin{equation}
    \theta \;\ge\; \alpha^{(t)} + \bigl(\beta^{(t)}\bigr)^\top y,
    \qquad t=1,2,\dots,T
    \label{eq:cflp_master_cut_form}
\end{equation}
where $t$ indexes the cut-generation iteration.

\noindent \textbf{Predict--and--Project.} Let $\lambda_i\ge 0$ denote multipliers for the client constraints $\sum_j x_{ij}\ge 1$. The proxy predicts $\tilde\lambda$ and projects to $\hat\lambda=(\tilde\lambda)_+$.

\noindent \textbf{Complete.} For a fixed projected $\hat\lambda\ge 0$, the strongest completion decomposes by facility. For facility $j$, define the continuous knapsack value
\begin{align}
\kappa_j(\lambda)
:=
\max\left\{
\sum_i (\lambda_{i}-d_ic_{ij})a_i
\;:\;
0\le a_i\le 1,\ \ \sum_i d_i a_i \le s_j
\right\}.
\label{eq:opt_kappa}
\end{align}
This is finite and attained for all $\lambda\ge 0$ because the feasible set is a compact polytope. It admits a closed-form completion: discard items with $\lambda_i-d_ic_{ij}\le 0$, sort the remainder by decreasing profit-to-weight ratio, and fill capacity until exhausted. The resulting knapsack-lifted cut coefficients for the $\theta$-master are
\begin{align}
\alpha(\hat\lambda)=\sum_i \hat\lambda_i,\qquad \beta_j(\hat\lambda)= - \kappa_j(\hat\lambda),\qquad
\theta \ge \alpha(\hat\lambda)+\sum_j \beta_j(\hat\lambda) y_j.
\label{eq:opt_cflp_cut}
\end{align}
Equivalently, writing the cut in terms of the total-objective epigraph variable $\eta:=\theta+\sum_j f_j y_j$ yields
\[
    \eta \ge \sum_i \hat\lambda_i + \sum_j \bigl(f_j-\kappa_j(\hat\lambda)\bigr)y_j.
 \]
Thus, for the CFLP above, the completion map is closed-form, inexpensive, and yields the strongest completion value for each predicted multiplier $\tilde\lambda$.
\end{exam}

\section{Proxy Feasibility Cuts}
\label{sec:neural_feasibility}
This section considers the case where the subproblem is infeasible at a trial solution $\bar y$. In such a case, Benders decomposition produces a feasibility cut that separates $\bar y$ while remaining valid for all feasible $y$. The primary challenge in learning feasibility certificates is \emph{scale invariance}. Classical Farkas certificates live in an unbounded cone and are defined only up to scaling. To resolve this, a \emph{slice-normalized} reformulation bounds the certificate set using a single linear constraint. Consequently, generating the proxy feasibility cut relies on the same predict-project-and-complete approach as the proxy optimality cut (Section~\ref{sec:neural_optimality}), requiring exactly one additional bounded constraint during the projection step.

\subsection{Slice Normalization}
\label{sec:slice_phase1}

For a candidate first-stage solution $\bar y$, the recourse feasibility system is
\begin{equation}
    \exists\, x \ge 0:\;\; Bx \ge d(\bar y),
    \tag{F}
    \label{eq:recourse_feasibility}
\end{equation}
with $d(y):=b-Ay$ as in~\eqref{eq:opt_sp_lp}. When~\eqref{eq:recourse_feasibility} is infeasible, classical Benders derives a feasibility cut from a Farkas certificate, i.e., an extreme ray of the dual cone $\mathcal R:=\{\lambda\ge 0:\ B^\top\lambda\le 0\}$. Rays are scale invariant, so a normalization must be chosen to anchor the certificate. Let $t\ge 0$ be a scalar slack and consider the uniform-relaxation problem
\begin{align}
    t^\star(\bar y)
    :=
    \min_{x \ge 0,\; t \ge 0}\;\;
    & t \nonumber\\
    \text{s.t.}\;\;
    & Bx + t\,\mathbf{e} \ge d(\bar y),
    \tag{Primal-slice}
    \label{eq:phase1_slice_primal}
\end{align}
where $\mathbf{e}$ is the all-ones vector. The dual of~\eqref{eq:phase1_slice_primal} is
\begin{align}
    \max_{\lambda}\;\; & d(\bar y)^\top \lambda \nonumber \\
    \text{s.t.}\;\;
    & B^\top \lambda \le 0, \qquad
      \lambda \ge 0,\qquad
      \mathbf{e}^\top \lambda \le 1,
    \tag{Dual-slice}
    \label{eq:phase1_slice_dual}
\end{align}
which optimizes over a bounded slice of the Farkas cone~$\mathcal R$.

\begin{theo}
\label{thm:phase1_slice_farkas}
The system~\eqref{eq:recourse_feasibility} is feasible if and only if $t^\star(\bar y)=0$. Moreover, if $t^\star(\bar y)>0$, then any dual optimizer $\lambda^\star$ of~\eqref{eq:phase1_slice_dual} satisfies $\mathbf{e}^\top \lambda^\star=1$ and yields the valid Benders feasibility cut
\begin{equation}
    (b-Ay)^\top \lambda^\star \le 0,
    \label{eq:feas_cut_slice}
\end{equation}
which separates $\bar y$.
\end{theo}
\begin{proof}
The equivalence $t^\star(\bar y)=0 \Leftrightarrow \eqref{eq:recourse_feasibility}$ is immediate: $t=0$ is feasible for~\eqref{eq:phase1_slice_primal} exactly when there exists $x\ge 0$ such that $Bx\ge d(\bar y)$. For $t^\star(\bar y)>0$, the primal is feasible and the dual has a nonempty, bounded feasible region, so strong duality holds and $t^\star(\bar y)=d(\bar y)^\top \lambda^\star>0$ at any dual optimizer. Complementary slackness on the uniform slack variable $t$ gives $t^\star(\bar y)(1-\mathbf{e}^\top \lambda^\star)=0$, hence $\mathbf{e}^\top \lambda^\star=1$. Validity of~\eqref{eq:feas_cut_slice} follows from $\lambda^\star\in\mathcal R$ and weak duality. Separation follows from $(b-A\bar y)^\top \lambda^\star=t^\star(\bar y)>0$.
\end{proof}

\begin{lemm}
\label{lem:slice_extreme_rays}
When $t^\star(\bar y)>0$, every optimal solution of~\eqref{eq:phase1_slice_dual} saturates $\mathbf{e}^\top\lambda=1$ and can be chosen at an extreme point of the slice $\mathcal{S}_{\mathbf e}:=\{\lambda:\lambda\ge 0,\;B^\top\lambda\le 0,\;\mathbf{e}^\top\lambda=1\}$, which corresponds, up to scaling, to an extreme ray of the Farkas cone $\mathcal R$.
\end{lemm}
\begin{proof}
Since $\lambda\ge 0$, every nonzero $\lambda\in\mathcal R$ satisfies $\mathbf{e}^\top\lambda>0$. Let $t^\star(\bar y)>0$ and let $\lambda^\star$ be optimal for~\eqref{eq:phase1_slice_dual}. If $\mathbf{e}^\top\lambda^\star<1$, then $\tilde\lambda:=\lambda^\star/(\mathbf{e}^\top\lambda^\star)$ remains feasible (since $\mathcal R$ is a cone) and satisfies $\mathbf{e}^\top\tilde\lambda=1$, with $d(\bar y)^\top\tilde\lambda>d(\bar y)^\top\lambda^\star$, contradicting optimality. Hence $\mathbf{e}^\top\lambda^\star=1$.
\par
Let $r$ generate an extreme ray of $\mathcal R$ and define $\lambda:=r/(\mathbf{e}^\top r)\in\mathcal{S}_{\mathbf e}$. If $\lambda=\alpha\lambda_1+(1-\alpha)\lambda_2$ with $\lambda_1,\lambda_2\in\mathcal{S}_{\mathbf e}$ and $\alpha\in(0,1)$, then $r=\alpha(\mathbf{e}^\top r)\lambda_1+(1-\alpha)(\mathbf{e}^\top r)\lambda_2$ with both terms in $\mathcal R$; extremality of the ray forces $\lambda_1$ and $\lambda_2$ colinear with $r$, and the normalization $\mathbf{e}^\top\lambda_1=\mathbf{e}^\top\lambda_2=1$ forces $\lambda_1=\lambda_2=\lambda$, so $\lambda$ is an extreme point. Conversely, if $\lambda\in\mathcal{S}_{\mathbf e}$ is an extreme point but not on an extreme ray of $\mathcal R$, write $\lambda=u_1+u_2$ for nonzero, non-colinear $u_1,u_2\in\mathcal R$; setting $\alpha_i:=\mathbf{e}^\top u_i$ and $\lambda_i:=u_i/\alpha_i\in\mathcal{S}_{\mathbf e}$ gives $\lambda=\alpha_1\lambda_1+\alpha_2\lambda_2$ with $\alpha_1+\alpha_2=1$, a nontrivial convex combination of two distinct points in $\mathcal{S}_{\mathbf e}$, contradicting extremality.
\end{proof}

\subsection{Predict--Project--and--Complete Pipeline}
\label{sec:feas_reduce_to_opt}

The slice-normalized certificate set in~\eqref{eq:phase1_slice_dual} is bounded. This makes the certificate-generation problem structurally similar to the optimality case in Section~\ref{sec:neural_optimality}, and allows the same proxy design to apply, with a single additional bounded constraint in the projection step.

To match the LP-recourse setting used elsewhere in the paper, assume the recourse set also contains $y$-independent structure constraints $Gx\ge g$ in addition to $x\ge 0$, and that set $\{x\ge 0:\,Gx\ge g\}$ is nonempty. A slice-normalized reformulation is then
\begin{align}
F(\bar y)
:=
\min_{x\ge 0,\; t\ge 0}\ \Bigl\{ t:\ Bx+t\,\mathbf{e} \ge d(\bar y),\ \ Gx\ge g \Bigr\},
\label{eq:phase1_lp}
\end{align}
which is always feasible and satisfies the augmented-recourse equivalence ``$\exists\,x\ge 0:\ Bx\ge d(\bar y),\ Gx\ge g$'' holds at $\bar y$ $\Leftrightarrow F(\bar y)=0$. This is the extension of Theorem~\ref{thm:phase1_slice_farkas} to the augmented slice, proved by the same complementary-slackness argument on the uniform slack $t$. The $Gx\ge g$ block is carried unchanged and does not affect the $t=0$ characterization. The dual of~\eqref{eq:phase1_lp} is
\begin{align}
\max_{\lambda,\mu}\quad & d(\bar y)^\top\lambda + g^\top\mu \label{eq:phase1_lp_dual}\\
\text{s.t.}\quad & B^\top\lambda + G^\top\mu \le 0, \nonumber\\
& \lambda\ge 0,\ \ \mu\ge 0,\ \ \mathbf{e}^\top\lambda \le 1. \nonumber
\end{align}
Any dual-feasible $(\lambda,\mu)$ yields the valid feasibility inequality
\begin{align}
0 \ge d(y)^\top\lambda + g^\top\mu .
\label{eq:feas_cut_lp}
\end{align}
The scalar certified value of the slice-Farkas certificate at $\bar y$ is
\begin{align}
L_{\mathrm{feas}}(\xi,\bar y):=d(\bar y)^\top\hat\lambda+g^\top\hat\mu,
\label{eq:proxy_Lfeas}
\end{align}
which lower-bounds $F(\bar y)$ by weak duality.  Moreover, if $L_{\mathrm{feas}}(\xi,\bar y)>0$, then $F(\bar y)>0$ and the recourse is infeasible at $\bar y$ (formalized as Proposition~\ref{prop:proxy_one_sided} below).

Denote by
\begin{equation}
V_1(B)\;:=\;\bigl\{\lambda\ge 0:\ B^\top\lambda\le 0,\ \mathbf{e}^\top\lambda\le 1\bigr\}
\label{eq:V1B}
\end{equation}
the slice of the Farkas cone used as the projection target for the predicted multiplier. For each $\hat\lambda\in V_1(B)$, the feasibility-mode completion in $\mu$ is the LP
\begin{align}
\hat\mu(\hat\lambda)\in\arg\max_{\mu\ge 0}\bigl\{g^\top\mu:\ G^\top\mu\le -B^\top\hat\lambda\bigr\}.
\label{eq:feas_completion_lp}
\end{align}

\begin{assum}
\label{ass:feas_regular_completion_lp}
For every $\hat\lambda\in V_1(B)$, the feasibility completion LP~\eqref{eq:feas_completion_lp} is feasible and attains an optimum.
\end{assum}
This holds whenever the $y$-independent recourse set $\{x\ge 0:\ Gx\ge g\}$ is nonempty: since $\hat\lambda\in V_1(B)$ gives $-B^\top\hat\lambda\ge 0$, the point $\mu=0$ is feasible for~\eqref{eq:feas_completion_lp}, and by LP duality the completion LP is bounded above precisely when $\{x\ge 0:\ Gx\ge g\}\neq\emptyset$, e.g., the CFLP formulation~\eqref{eq:cflp_primal}.

\begin{theo}
\label{thm:feas_completion_existence_lp}
Under Assumption~\ref{ass:feas_regular_completion_lp}, for any network output $\tilde\lambda$ the projected multiplier $\hat\lambda=\Pi_{V_1(B)}(\tilde\lambda)$ admits a completion $\hat\mu$ given by~\eqref{eq:feas_completion_lp}, and the pair $(\hat\lambda,\hat\mu)$ is dual-feasible for~\eqref{eq:phase1_lp_dual}.
\end{theo}
\begin{proof}
By construction of the projection, $\hat\lambda\in V_1(B)$, so $\hat\lambda\ge 0$, $B^\top\hat\lambda\le 0$, and $\mathbf{e}^\top\hat\lambda\le 1$.  By Assumption~\ref{ass:feas_regular_completion_lp}, problem~\eqref{eq:feas_completion_lp} attains an optimum $\hat\mu\ge 0$ satisfying $G^\top\hat\mu\le -B^\top\hat\lambda$.  Rearranging gives $B^\top\hat\lambda+G^\top\hat\mu\le 0$, which together with the slice normalization is exactly the dual-feasibility region of~\eqref{eq:phase1_lp_dual}.
\end{proof}

\begin{theo}
\label{thm:feas_strongest_completion_lp}
For any fixed $\hat\lambda\in V_1(B)$, every optimizer $\hat\mu$ of~\eqref{eq:feas_completion_lp} maximizes the certificate value $d(\bar y)^\top\hat\lambda+g^\top\hat\mu$ among all dual-feasible completions of $\hat\lambda$ for~\eqref{eq:phase1_lp_dual}.
\end{theo}
\begin{proof}
The objective $g^\top\mu$ is linear and $d(\bar y)^\top\hat\lambda$ is constant in $\mu$, so any $\arg\max$ over the completion set in~\eqref{eq:feas_completion_lp} also maximizes $d(\bar y)^\top\hat\lambda+g^\top\mu$.
\end{proof}
Comparing~\eqref{eq:phase1_lp_dual} with the proxy optimality dual~\eqref{eq:opt_sp_lp_dual} in Section~\ref{sec:neural_optimality}, the structure is the same: a nonnegative multiplier $\mu$ coupled to $\lambda$ through a linear inequality, and a linear objective that becomes the cut's certified bound. The only new feature is the additional \emph{slice-normalization} constraint $\mathbf{e}^\top\lambda\le 1$, which removes the scale ambiguity of ray certificates. Therefore, feasibility-cut learning uses the same building blocks as the proxy optimality cuts: the proxy predicts $\tilde\lambda$, the projection enforces $\hat\lambda=\Pi_{V_1(B)}(\tilde\lambda)$, and the completion in $\mu$ proceeds via~\eqref{eq:feas_completion_lp}.  Validity of the resulting cut is then a direct consequence of weak duality and Theorem~\ref{thm:feas_completion_existence_lp}:

\begin{prop}
\label{prop:feas_cut_valid}
Let $(\hat\lambda,\hat\mu)$ be any dual-feasible pair for~\eqref{eq:phase1_lp_dual}.  Then the feasibility inequality~\eqref{eq:feas_cut_lp} is valid for every $y$ whose recourse~\eqref{eq:recourse_feasibility} is feasible.
\end{prop}
\begin{proof}
If recourse is feasible at $y$ then $F(y)=0$ in~\eqref{eq:phase1_lp}. Weak duality between~\eqref{eq:phase1_lp} and~\eqref{eq:phase1_lp_dual} gives $d(y)^\top\hat\lambda+g^\top\hat\mu\le F(y)=0$, which is exactly~\eqref{eq:feas_cut_lp}.
\end{proof}

Proposition~\ref{prop:feas_cut_valid} guarantees the emitted cut is never violated by a recourse-feasible $y$.  The converse direction, i.e., that a positive certified value certifies recourse infeasibility, is the one-sided guarantee of the slice-Farkas certificate.

\begin{prop}
\label{prop:proxy_one_sided}
If $L_{\mathrm{feas}}(\xi,\bar y)>0$ at~\eqref{eq:proxy_Lfeas}, then the recourse problem~\eqref{eq:opt_sp_lp} is infeasible at $\bar y$.  The converse does not hold: a nonpositive certificate value does not certify feasibility.
\end{prop}
\begin{proof}
Let $(\hat\lambda,\hat\mu)$ be dual-feasible for~\eqref{eq:phase1_lp_dual}.  By weak duality, $L_{\mathrm{feas}}(\xi,\bar y)=d(\bar y)^\top\hat\lambda+g^\top\hat\mu\le F(\bar y)$.  If $L_{\mathrm{feas}}>0$ then $F(\bar y)>0$, by the augmented-recourse equivalence established below~\eqref{eq:phase1_lp} (the extension of Theorem~\ref{thm:phase1_slice_farkas} to the $Gx\ge g$ block), $F(\bar y)>0$ iff the recourse~\eqref{eq:opt_sp_lp} is infeasible at $\bar y$.  The non-converse follows since a particular dual-feasible certificate can be too weak to attain $F(\bar y)$.
\end{proof}

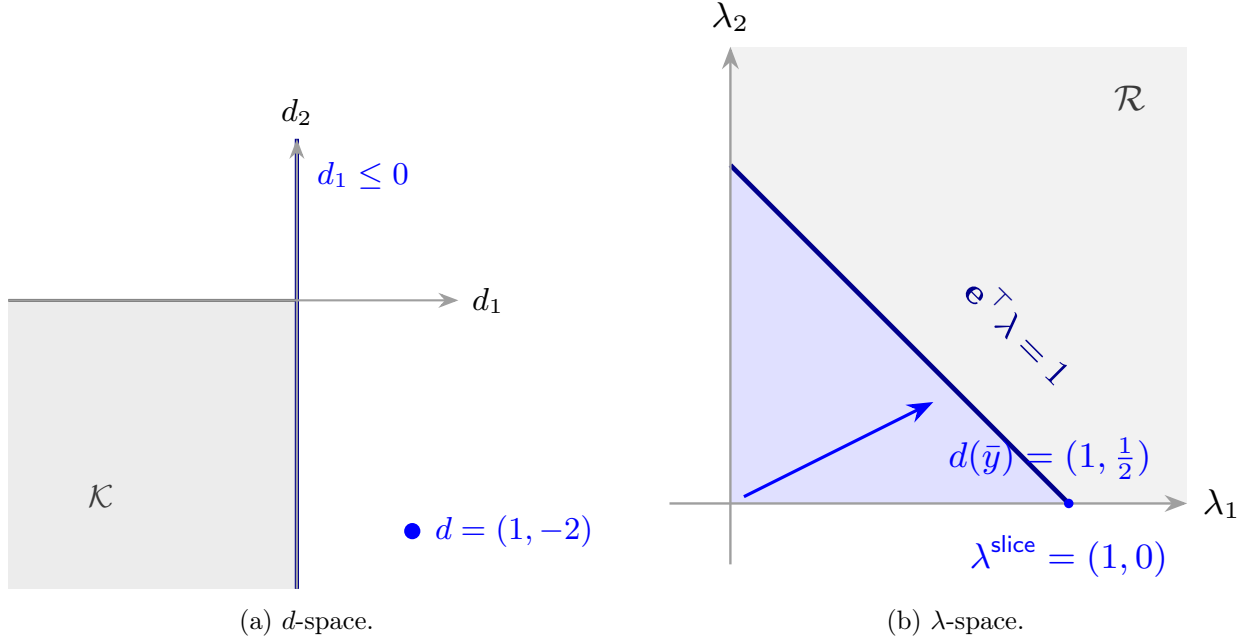
\begin{figure}[!t]
    \centering
    \begin{subfigure}[t]{0.48\textwidth}
        \centering
        \resizebox{\linewidth}{!}{%
        \begin{tikzpicture}[
            scale=1.15,
            >=Stealth,
            font=\footnotesize\sffamily,
            every node/.style={font=\footnotesize\sffamily}
        ]
            \begin{scope}
                \clip (-2.5,-2.5) rectangle (1.4,1.4);
                \fill[gray!15] (-2.5,-2.5) rectangle (0,0);
            \end{scope}
            \draw[gray!55!black, thick] (-2.5,0) -- (0,0);
            \draw[gray!55!black, thick] (0,-2.5) -- (0,0);
            \draw[blue!55!black, very thick] (0,-2.5) -- (0,1.4);
            \draw[->, gray!75, semithick] (-2.5,0) -- (1.4,0) node[right, black] {$d_1$};
            \draw[->, gray!75, semithick] (0,-2.5) -- (0,1.4) node[above, black] {$d_2$};
            \node[gray!45!black, font=\footnotesize\sffamily\itshape] at (-1.7,-1.7) {$\mathcal{K}$};
            \node[blue, anchor=south west] at (0.06,0.85) {$d_1\le 0$};
            \fill[blue] (1,-2) circle (2pt);
            \node[blue, anchor=west] at (1.08,-2) {$d=(1,-2)$};
        \end{tikzpicture}%
        }
        \caption{$d$-space.}
    \end{subfigure}
    \hfill
    \begin{subfigure}[t]{0.48\textwidth}
        \centering
        \resizebox{\linewidth}{!}{%
        \begin{tikzpicture}[
            scale=3.0,
            >=Stealth,
            font=\footnotesize\sffamily,
            every node/.style={font=\footnotesize\sffamily}
        ]
            \begin{scope}
                \clip (-0.22,-0.22) rectangle (1.35,1.35);
                \fill[gray!10] (0,0) rectangle (1.35,1.35);
            \end{scope}
            \fill[blue!12] (0,0) -- (1,0) -- (0,1) -- cycle;
            \draw[blue!55!black, very thick] (1,0) -- (0,1);
            \draw[->, gray!75, semithick] (-0.18,0) -- (1.35,0) node[right, black] {$\lambda_1$};
            \draw[->, gray!75, semithick] (0,-0.18) -- (0,1.35) node[above, black] {$\lambda_2$};
            \node[gray!45!black, font=\footnotesize\sffamily\itshape] at (1.18,1.20) {$\mathcal{R}$};
            \node[blue!55!black, rotate=-45, anchor=south west] at (0.62,0.62) {$\mathbf{e}^\top\lambda=1$};
            \draw[->, blue, thick] (0.04,0.02) -- (0.60,0.30);
            \node[blue, anchor=north west] at (0.60,0.24) {$d(\bar y)=(1,\tfrac12)$};
            \fill[blue] (1,0) circle (0.014);
            \node[blue, anchor=north] at (1,-0.05) {$\lambda^{\text{slice}}=(1,0)$};
        \end{tikzpicture}%
        }
        \caption{$\lambda$-space.}
    \end{subfigure}
    \caption{Slice-normalized geometry: the slice certificate $\lambda^{\text{slice}}=(1,0)$ maximizes $d(\bar y)^\top\lambda$ over $\{\lambda\ge 0,\,\mathbf{e}^\top\lambda\le 1\}$ (right) and induces the facet cut $d_1\le 0$ separating the infeasible $d=(1,-2)$ from $\mathcal K$ (left).}
    \label{fig:slice-vs-box}
\end{figure}

\begin{exam}\label{ex:box-mixture}
Take $m=2$, $B=-I_2$, and $X=\mathbb{R}^2_+$. The recourse-feasible cone in the right-hand-side space is
\[
\mathcal{K}=\{d\in\mathbb{R}^2:\; d_1\le 0,\; d_2\le 0\},
\]
whose facet inequalities are $d_1\le 0$ and $d_2\le 0$. The certificate cone is
\(
\mathcal{R}=\{\lambda\ge 0,\;B^\top \lambda\le 0\}=\mathbb{R}^2_+.
\)
Under slice normalization, the proxy optimizes over $\mathcal{S}_{\mathbf e}=\{\lambda\in\mathcal{R}:\ \mathbf{e}^\top\lambda\le 1\}$. For $d(\bar y)=(1,\tfrac12)$, Lemma~\ref{lem:slice_extreme_rays} implies that an optimal certificate can be chosen at an extreme point of $\mathcal{S}_{\mathbf e}$, equivalently, on an extreme ray of $\mathcal{R}$. The unique optimizer here is $\lambda^{\text{slice}}=(1,0)$, and the induced feasibility cut is
\[
d(y)^\top\lambda^{\text{slice}}\le 0
\quad\Longleftrightarrow\quad
d_1(y)\le 0,
\]
which is exactly a \emph{facet} of $\mathcal{K}$. Figure~\ref{fig:slice-vs-box} shows this certificate in $\lambda$-space and the corresponding facet cut in $d$-space.
\end{exam}

The slice constraint $\mathbf{e}^\top\lambda\le 1$ removes the scaling ambiguity of ray certificates, lets the proxy target a bounded certificate set, and makes the separation value $d(\bar y)^\top\lambda$ comparable across instances and RMP solutions. More generally, one may normalize with another positive linear functional $r^\top\lambda\le 1$. Appendix~\ref{A} discusses box normalization as a contrasting alternative under which the optimizer may yield a valid but non-facet cut.

\subsection{Deployment repair}\label{sec:feas_repair}

\noindent The proxy's slice-Farkas cut~\eqref{eq:feas_cut_lp} can be shallow at $\bar y$ when the predicted $\hat\lambda$ under-prices the recourse, so the master can converge at an incumbent that is not strictly feasible for the original problem. When a strictly feasible design is required, a problem-specific repair procedure is applied to the returned design. The repair strategy depends on the structure of the problem: when feasibility is monotone in the design variables (e.g., as in network design where opening more arcs improves feasibility), a greedy procedure guided by the proxy's dual prices can restore feasibility. The repair is therefore optional and problem-specific, invoked only when an exactly feasible design must be returned.

At each separation invocation, the feasibility certificate~\eqref{eq:proxy_Lfeas} is evaluated first. If $L_{\mathrm{feas}}(\xi,\bar y)>0$, the associated slice-Farkas inequality~\eqref{eq:feas_cut_lp} is emitted (the recourse is infeasible at $\bar y$ by Proposition~\ref{prop:proxy_one_sided}), otherwise, the optimality certificate is evaluated, and the optimality cut~\eqref{eq:opt_cut_lp} is added when violated. No recourse LP is solved during proxy deployment. An exact solve is invoked only afterward, outside the timed comparison, to compute the true optimality gap.

\section{The Proxy Benders Decomposition Framework}
\label{sec:proxy_benders}

Sections~\ref{sec:neural_optimality} and \ref{sec:neural_feasibility} established the predict--project--and--complete layer for generating proxy optimality cuts and, when the recourse can be infeasible, proxy feasibility cuts via the slice-normalized reformulation. This section presents the \emph{Proxy-BD} framework: an optimization proxy that replaces repeated exact Benders subproblem optimization by a forward pass through a certification layer, producing provably valid Benders cuts without requiring exact subproblem separation during deployment. Although Proxy-BD is first introduced within the classical Benders framework to isolate and analyze the proxy separation mechanism, its main motivation is modern large-scale decomposition schemes such as B\&BC algorithms, where subproblem separation is repeatedly invoked throughout the branch-and-bound tree \citep{moreno2019branch,satici2026branch}. In such settings, the cumulative cost of exact recourse solves becomes substantial, making amortized proxy inference particularly advantageous.

This section proceeds as follows. Section~\ref{sec:opt_block3} presents the self-supervised training framework. Section~\ref{sec:proxy_unified} describes the unified proxy architecture that handles both optimality and feasibility cases. Sections~\ref{sec:proxy_offline} and~\ref{sec:proxy_online} present the offline training pipeline and the online deployment procedure, respectively. Section~\ref{sec:proxy_validity} highlights theoretical guarantees that ensure every proxy-generated cut is valid.

\subsection{Self-supervised Learning}
\label{sec:opt_block3}

Self-supervised learning is a central component of Proxy-BD. Unlike supervised learning, which would require optimal dual labels $\lambda^\star$, Proxy-BD is trained by maximizing a certified dual bound. This choice is motivated by the fact that the objective of Proxy-BD is not to recover a particular optimal dual solution, but to generate strong Benders cuts. Furthermore, since Benders subproblems are often dual-degenerate, multiple dual-optimal certificates may generate identical cuts, making dual vectors ambiguous supervision targets. A supervised loss would therefore optimize dual imitation rather than decomposition performance. A \emph{Benders state} is defined as the pair $(\xi,\bar y)$, where $\xi$ denotes the instance data features encoding $(A,B,G,b,g,q)$ and $\bar y$ is the current master solution at which the algorithm attempts to generate a violated cut. During training, oracle Benders runs are used only to generate representative Benders states $(\xi,\bar y)$ and not to provide dual supervision labels. The certification layer returns a dual-feasible pair $(\hat\lambda,\hat\mu)$ and the certified scalar value
\begin{align}
L(\xi,\bar y) := d(\bar y)^\top \hat\lambda + g^\top \hat\mu,
\label{eq:opt_training_bound}
\end{align}
which coincides with $L$ in \eqref{eq:proxy_Lopt}. By weak duality, $L(\xi,\bar y)$ is the right-hand side of the certified Benders cut generated by the proxy. Consequently, maximizing $L$ directly maximizes the strength of the generated cuts. The self-supervised objective therefore optimizes the quantity that drives decomposition performance, rather than the recovery of a specific dual certificate. The combined training objective is given in \eqref{eq:proxy_train_loss}.

The completion problem \eqref{eq:opt_completion_lp} admits the primal representation
\[
\min_{x\ge 0,\;Gx\ge g}\ (q-B^\top\hat\lambda)^\top x.
\]
The certified value $L$ is concave in $\hat\lambda$, as it is the optimal value of a linear program whose objective coefficients depend affinely on $\hat\lambda$. Let $\hat x$ denote an optimizer of the completion primal. Then a supergradient of $L$ with respect to the projected multiplier is
\begin{align}
\partial^+_{\hat\lambda} L(\xi,\bar y)\ni d(\bar y) - B\hat x,
\label{eq:opt_subgrad_lambda}
\end{align}
i.e., the residual of the linking constraints evaluated at the completion solution; equivalently, $B\hat x-d(\bar y)$ is a subgradient of the convex training loss $-L$. This is the dual-ascent mechanism emphasized by DLL \citep{tanneau2024dual}: training network parameters to yield larger certified values, even when no dual-optimal labels are provided.

At inference time, no parameter updates are performed. Given $\bar y$, the network is evaluated, projection and, when needed, completion are applied, and the resulting certified optimality cut is added to the master.

\subsection{Proxy Architecture}
\label{sec:proxy_unified}

Proxy-BD consists of three interconnected components:
(i) an inference module that maps a Benders state
$(\xi,\bar y)$ to a partial dual prediction, (ii) a certification layer that transforms the prediction into a dual-feasible certificate through projection and completion, and (iii) a deployment procedure that integrates certified cuts into a Benders decomposition framework. The certification mechanisms for optimality and feasibility cuts were established in Sections~\ref{sec:neural_optimality} and~\ref{sec:neural_feasibility}. The present section focuses on how these components interact during training and deployment.

\noindent \textbf{Predict.} The proxy network maps a Benders state to a partial dual prediction:
\begin{align}
    \tilde\lambda = g_\theta(\xi,\bar y)\in\mathbb{R}^{n_\lambda}.
    \label{eq:proxy_network}
\end{align}

\noindent \textbf{Project.} The prediction is projected onto the nonnegative orthant:
\begin{align}
    \hat\lambda = \Pi_{\{\lambda\ge 0\}}(\tilde\lambda) = (\tilde\lambda)_+,
    \label{eq:proxy_proj}
\end{align}

\noindent \textbf{Complete.} The remaining dual variables are recovered through the completion problem:
\begin{align}
    \hat\mu(\hat\lambda) \in \arg\max_{\mu\ge 0}\
        \bigl\{ g^\top\mu:\ G^\top\mu \le q - B^\top\hat\lambda \bigr\},
    \label{eq:proxy_completion_opt}
\end{align}
which is often available in closed form for structured recourse models (see  Example~\ref{sec:opt_cflp_example}).

\noindent \textbf{Certification.} The certification layer returns the certified scalar value
\begin{align}
    L(\xi,\bar y) = d(\bar y)^\top \hat\lambda + g^\top \hat\mu,
    \label{eq:proxy_Lopt}
\end{align}
where $d(\bar y):=b-A\bar{y}$. The resulting cut coefficients are
\begin{align}
    \alpha = b^\top\hat\lambda + g^\top\hat\mu,
    \qquad
    \beta = -A^\top\hat\lambda,
    \label{eq:proxy_coeff_opt}
\end{align}
defining the optimality cut $\theta \ge \alpha + \beta^\top y$.

Proxy-BD operates through two phases: an offline phase that generates Benders states and trains the proxy through self-supervised learning, and an online phase that uses the trained proxy to generate cuts inside Benders decomposition.

\subsubsection{Offline Phase}
\label{sec:proxy_offline}

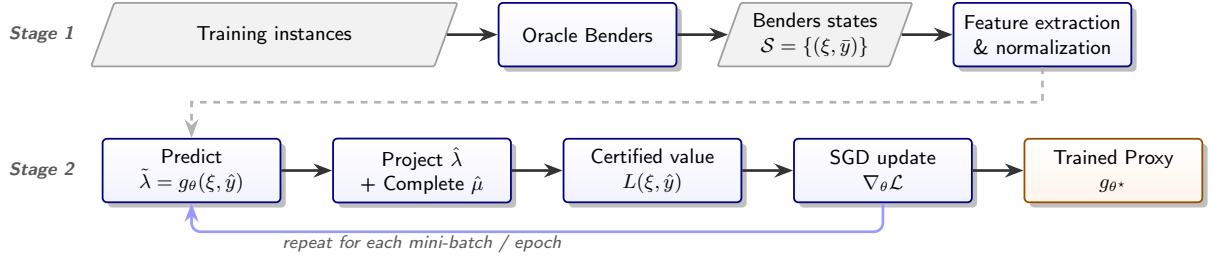
\begin{figure}[!t]
\centering
\resizebox{0.97\textwidth}{!}{%
\begin{tikzpicture}[
    node distance=0.7cm and 0.8cm,
    auto,
    >=Stealth,
    thick,
    font=\footnotesize\sffamily,
    process/.style={
        rectangle,
        draw=blue!50!black,
        top color=white,
        bottom color=blue!5,
        rounded corners=2pt,
        minimum width=2.8cm,
        minimum height=1.0cm,
        align=center,
        drop shadow
    },
    data/.style={
        trapezium,
        trapezium left angle=70,
        trapezium right angle=110,
        draw=gray!75,
        fill=gray!10,
        rounded corners=1pt,
        minimum width=2.6cm,
        minimum height=1.0cm,
        align=center
    },
    output/.style={
        rectangle,
        draw=orange!60!black,
        top color=white,
        bottom color=orange!8,
        rounded corners=2pt,
        minimum width=2.8cm,
        minimum height=1.0cm,
        align=center,
        drop shadow
    },
    line/.style={draw, ->, very thick, color=black!80},
    phaselabel/.style={font=\scriptsize\sffamily\itshape, text=gray!70!black}
]
    \node[phaselabel] (lbl1) {\textbf{Stage 1}};
    \node[data, right=0.3cm of lbl1] (base)
      {Training instances};
    \node[process, right=of base] (oracle)
      {Oracle Benders};
    \node[data, right=of oracle] (states)
      {Benders states\\$\mathcal{S}=\{(\xi,\bar y)\}$};
    \node[process, right=of states] (feat)
      {Feature extraction\\\& normalization};

    \draw[line] (base) -- (oracle);
    \draw[line] (oracle) -- (states);
    \draw[line] (states) -- (feat);

    \node[phaselabel, below=1.6cm of lbl1] (lbl2) {\textbf{Stage 2}};
    \node[process, right=0.3cm of lbl2] (predict)
      {Predict\\$\tilde\lambda=g_\theta(\xi,\hat y)$};
    \node[process, right=of predict] (projcomp)
      {Project $\hat\lambda$\\+ Complete $\hat\mu$};
    \node[process, right=of projcomp] (bound)
        {Certified value\\$L(\xi,\hat y)$};
    \node[process, right=of bound] (sgd)
      {SGD update\\$\nabla_\theta\mathcal{L}$};
    \node[output, right=of sgd] (model)
      {Trained Proxy \\$g_{\theta^\star}$};

    \draw[line, dashed, color=gray!60] (feat.south) -- ++(0,-0.55) -| (predict.north);
    \draw[line] (predict) -- (projcomp);
    \draw[line] (projcomp) -- (bound);
    \draw[line] (bound) -- (sgd);
    \draw[line] (sgd) -- (model);

    \draw[line, color=blue!40, rounded corners=4pt]
      (sgd.south) -- ++(0,-0.45) -| (predict.south);
    \node[phaselabel, below=0.35cm of projcomp]
      {repeat for each mini-batch / epoch};
\end{tikzpicture}
}
\caption{Offline pipeline for Proxy-BD. Stage~1: Oracle Benders is run on training instances to collect states $(\xi,\bar y)$. Stage~2: the proxy is trained self-supervised by minimizing $\mathcal{L}(\theta)$ in \eqref{eq:proxy_train_loss}, with gradients flowing through the completion map via \eqref{eq:opt_subgrad_lambda}.}
\label{fig:proxy_offline}
\end{figure}

The offline phase produces a trained proxy $g_{\theta^\star}$. It consists of two stages (see Figure~\ref{fig:proxy_offline}).

\noindent \textbf{Stage~1: State Generation.} An oracle Benders procedure with exact subproblem separation is run on each training instance to collect Benders states $\{(\xi_n, \bar y_n^{(t)})\}$, where $\bar y_n^{(t)}$ is the separation point at iteration~$t$. These states form the dataset $\mathcal{S}$.

\noindent \textbf{Stage~2: Self-supervised Training.} The proxy parameters $\theta$ are optimized by maximizing the certified dual bound over $\mathcal{S}$. The self-supervised training maximizes certified bounds over the collected Benders states, encouraging the proxy to generate strong dual certificates throughout the decomposition process. The loss takes the general form:
\begin{align}
    \mathcal{L}(\theta)
    &=
    \mathbb{E}_{(\xi,\bar y)\in\mathcal{S}}\bigl[
        \ell(\xi,\bar y; \theta)
    \bigr],
    \label{eq:proxy_train_loss}
\end{align}
where $\ell$ promotes strong certified cuts over the distribution of Benders states encountered during training. The loss is defined through the optimality certificate $L$. The specific loss formulations address problem-specific challenges, such as invariances, and are detailed in the appendices.

The training is self-supervised in that no optimal dual solutions are required. The proxy learns to generate strong cuts using only the completion's certificate value and gradient. The gradient signal comes from the completion subgradient $\partial^+_{\hat\lambda} L_\theta \ni d(\bar y) - B\hat x$ (cf.\ \eqref{eq:opt_subgrad_lambda}), which flows through the projection layer and into the network parameters via backpropagation.

\subsubsection{Online Phase}
\label{sec:proxy_online}

At deployment, the trained proxy is queried in place of each exact subproblem solve within a classical outer-loop Benders framework. At each iteration~$t$:
\begin{enumerate}
    \item Solve the RMP to obtain $(\bar y^{(t)},\bar\theta^{(t)})$.
    \item Evaluate the proxy at the solution $\bar y^{(t)}$ using predict--project--and--complete in \eqref{eq:proxy_network}--\eqref{eq:proxy_completion_opt}.
    \item Form the certified cut coefficients \eqref{eq:proxy_coeff_opt}.
    \item If the cut is violated at $(\bar y^{(t)},\bar\theta^{(t)})$, add it to the RMP and go to Step~1. Otherwise, terminate.
\end{enumerate}

Figure~\ref{fig:proxy_online} illustrates Proxy-BD within a classical outer-loop Benders framework, where the restricted master problem and the proxy alternate. In modern decomposition frameworks such as B\&BC algorithms, the same Proxy-BD deployment mechanism is invoked repeatedly at nodes of the branch-and-bound tree, replacing repeated exact recourse optimization with amortized proxy inference. Figure~\ref{fig:bbc_loop} illustrates this setup.

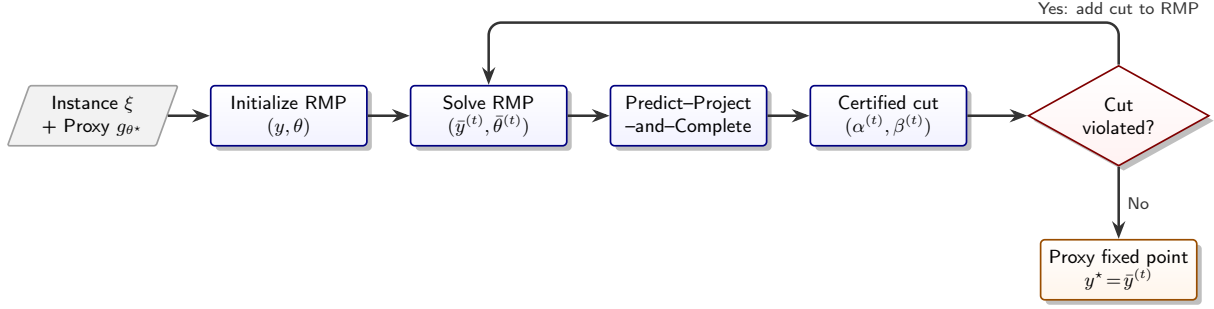
\begin{figure}[!t]
\centering
\resizebox{0.97\textwidth}{!}{%
\begin{tikzpicture}[
    node distance=0.7cm and 0.7cm,
    auto,
    >=Stealth,
    thick,
    font=\footnotesize\sffamily,
    process/.style={
        rectangle,
        draw=blue!50!black,
        top color=white,
        bottom color=blue!5,
        rounded corners=2pt,
        minimum width=2.6cm,
        minimum height=1.0cm,
        align=center,
        drop shadow
    },
    data/.style={
        trapezium,
        trapezium left angle=70,
        trapezium right angle=110,
        draw=gray!75,
        fill=gray!10,
        rounded corners=1pt,
        minimum width=2.4cm,
        minimum height=1.0cm,
        align=center
    },
    decision/.style={
        diamond,
        draw=red!50!black,
        top color=white,
        bottom color=red!5,
        minimum width=2.4cm,
        minimum height=1.6cm,
        align=center,
        drop shadow,
        aspect=1.8
    },
    output/.style={
        rectangle,
        draw=orange!60!black,
        top color=white,
        bottom color=orange!8,
        rounded corners=2pt,
        minimum width=2.6cm,
        minimum height=1.0cm,
        align=center,
        drop shadow
    },
    line/.style={draw, ->, very thick, color=black!80},
    phaselabel/.style={font=\scriptsize\sffamily\itshape, text=gray!70!black}
]
    \node[data] (input) {Instance $\xi$\\+ Proxy $g_{\theta^\star}$};
    \node[process, right=of input] (init) {Initialize RMP \\ $(y,\theta)$};
    \node[process, right=of init] (rmp) {Solve RMP\\$(\bar y^{(t)},\bar\theta^{(t)})$};
    \node[process, right=of rmp] (proxy)
      {Predict--Project\\--and--Complete};
    \node[process, right=of proxy] (cut) {Certified cut\\$(\alpha^{(t)},\beta^{(t)})$};
    \node[decision, right=1.0cm of cut] (violated) {Cut\\violated?};
    \node[output, below=1.2cm of violated] (sol) {Proxy fixed point\\$y^\star\!=\!\bar y^{(t)}$};

    \draw[line] (input) -- (init);
    \draw[line] (init) -- (rmp);
    \draw[line] (rmp) -- (proxy);
    \draw[line] (proxy) -- (cut);
    \draw[line] (cut) -- (violated);

    \draw[line, rounded corners=6pt]
      (violated.north) -- ++(0,0.7)
      node[above, font=\scriptsize\sffamily]{Yes: add cut to RMP}
      -| (rmp.north);

    \draw[line] (violated.south) -- (sol)
      node[midway, right, font=\scriptsize\sffamily]{No};

\end{tikzpicture}
}
\caption{%
Online Proxy-BD loop. At each iteration, the RMP is solved, and the proxy generates a certified cut via predict--project--complete. The loop terminates at a proxy fixed point (Definition~\ref{def:proxy_fixedpoint}).%
}
\label{fig:proxy_online}
\end{figure}

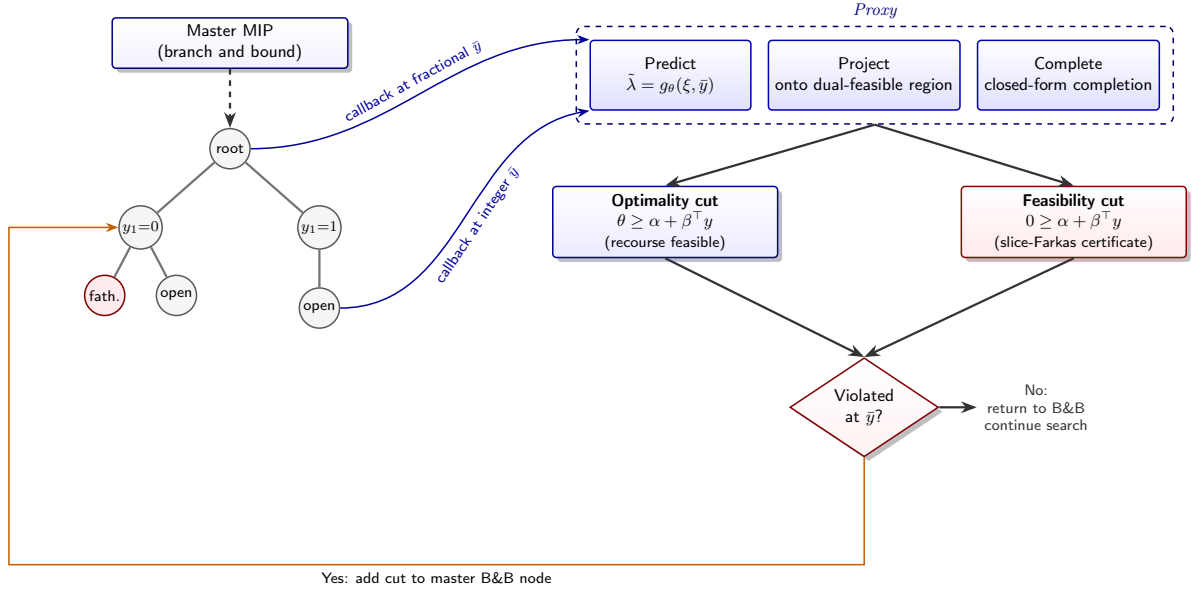
\begin{figure}[!t]
\centering
\resizebox{0.95\textwidth}{!}{%
\begin{tikzpicture}[
    node distance=0.7cm and 0.6cm,
    auto, >=Stealth, thick, font=\footnotesize\sffamily,
    process/.style={
        rectangle, draw=blue!50!black, top color=white, bottom color=blue!5,
        rounded corners=2pt, minimum width=4.4cm, minimum height=0.9cm,
        align=center, drop shadow
    },
    sub/.style={
        rectangle, draw=blue!70!black, top color=white, bottom color=blue!12,
        rounded corners=2pt, minimum width=3.0cm, minimum height=1.3cm,
        align=center
    },
    optcut/.style={
        rectangle, draw=blue!60!black, top color=white, bottom color=blue!8,
        rounded corners=2pt, minimum width=4.2cm, minimum height=1.1cm,
        align=center, drop shadow
    },
    feascut/.style={
        rectangle, draw=red!55!black, top color=white, bottom color=red!8,
        rounded corners=2pt, minimum width=4.2cm, minimum height=1.1cm,
        align=center, drop shadow
    },
    repair/.style={
        rectangle, draw=teal!60!black, top color=white, bottom color=teal!8,
        rounded corners=2pt, minimum width=4.4cm, minimum height=0.9cm,
        align=center, drop shadow
    },
    treenode/.style={
        circle, draw=gray!70!black, fill=gray!8,
        minimum size=0.75cm, inner sep=1pt, align=center,
        font=\scriptsize\sffamily
    },
    fath/.style={
        circle, draw=red!50!black, fill=red!8,
        minimum size=0.75cm, inner sep=1pt, align=center,
        font=\scriptsize\sffamily
    },
    decision/.style={
        diamond, draw=red!50!black, top color=white, bottom color=red!5,
        minimum width=1.8cm, minimum height=1.6cm,
        align=center, drop shadow, aspect=1.5
    },
    output/.style={
        rectangle, draw=orange!60!black, top color=white, bottom color=orange!8,
        rounded corners=2pt, minimum width=4.4cm, minimum height=0.9cm,
        align=center, drop shadow
    },
    line/.style={draw, ->, very thick, color=black!80},
    treeedge/.style={draw, very thick, color=black!55},
    cb/.style={draw, ->, thick, color=blue!60!black},
    rt/.style={draw, ->, thick, color=orange!75!black}
]
    \node[process] (master) {Master MIP\\(branch and bound)};

    \node[treenode, below=1.1cm of master] (root) {root};
    \node[treenode, below left=0.9cm and 1.1cm of root] (cl) {$y_1{=}0$};
    \node[treenode, below right=0.9cm and 1.1cm of root] (cr) {$y_1{=}1$};
    \node[fath, below left=0.7cm and 0.1cm of cl] (clll) {fath.};
    \node[treenode, below right=0.7cm and 0.1cm of cl] (cllr) {open};
    \node[treenode, below=0.7cm of cr] (crl) {open};

    \draw[treeedge] (root) -- (cl);
    \draw[treeedge] (root) -- (cr);
    \draw[treeedge] (cl) -- (clll);
    \draw[treeedge] (cl) -- (cllr);
    \draw[treeedge] (cr) -- (crl);
    \draw[line, dashed] (master) -- (root);

    \node[sub, right=4.5cm of master, yshift=-0.6cm] (predict) {Predict\\$\tilde\lambda = g_\theta(\xi,\bar y)$};
    \node[sub, right=0.3cm of predict] (project) {Project\\onto dual-feasible region};
    \node[sub, right=0.3cm of project] (complete) {Complete\\closed-form completion};

    \node[draw=blue!40!black, dashed, rounded corners=4pt, inner sep=7pt,
          fit=(predict)(project)(complete)] (proxybox) {};
    \node[blue!50!black, font=\footnotesize\itshape] at ([yshift=8pt]proxybox.north) {Proxy};

    \draw[cb] (root.east) to[out=0, in=180] node[midway, above, sloped, font=\scriptsize\sffamily]{callback at fractional $\bar y$} (predict.north west);
    \draw[cb] (crl.east) to[out=0, in=190] node[midway, below, sloped, font=\scriptsize\sffamily]{callback at integer $\bar y$} (predict.south west);

    \node[optcut, below=1.4cm of predict, xshift=-0.1cm] (cutopt)
        {\textbf{Optimality cut}\\$\theta \ge \alpha + \beta^\top y$\\{\scriptsize(recourse feasible)}};
    \node[feascut, below=1.4cm of complete, xshift=0.1cm] (cutfeas)
        {\textbf{Feasibility cut}\\$0 \ge \alpha + \beta^\top y$\\{\scriptsize(slice-Farkas certificate)}};

    \draw[line] (proxybox.south) -- (cutopt.north);
    \draw[line] (proxybox.south) -- (cutfeas.north);

    \node[decision, below=4.6cm of project] (viol) {Violated\\at $\bar y$?};
    \draw[line] (cutopt.south) -- (viol.north);
    \draw[line] (cutfeas.south) -- (viol.north);

    \coordinate (ydown) at ($(viol.south)+(0,-2.0)$);
    \coordinate (yleft) at ($(clll.west)+(-1.4,0)$);
    \coordinate (ybl)   at (yleft |- ydown);
    \coordinate (ytl)   at (yleft |- cl.west);
    \draw[rt] (viol.south) -- (ydown) -- (ybl) -- (ytl) -- (cl.west);
    \path (ydown) -- (ybl) node[midway, below, font=\scriptsize\sffamily]{Yes: add cut to master B\&B node};

    \draw[line] (viol.east) -- ++(0.7,0)
        node[right, font=\scriptsize\sffamily, align=center]{No:\\return to B\&B\\continue search};
\end{tikzpicture}
}
\caption{Online B\&BC Proxy-BD loop. The master MIP is solved once. The proxy is invoked as a separation callback at selected fractional or integer nodes. If the certified cut is violated, it is added to the current B\&B node; otherwise, control returns to the branch-and-bound search.}
\label{fig:bbc_loop}
\end{figure}

\subsection{Theoretical Guarantees}
\label{sec:proxy_validity}

Sections~\ref{sec:neural_optimality} and~\ref{sec:neural_feasibility} established the local certification properties of the predict--project--and--complete mechanism, including the existence of dual-feasible completions, strongest-completion optimality for fixed projected multipliers, and construction of certified dual inequalities. This section interprets these properties at the framework level and studies their implications for Proxy-BD deployment. The following proposition establishes that every proxy-generated inequality defines a valid Benders cut.

\begin{prop}
\label{prop:proxy_cut_valid}
Every dual-feasible pair $(\hat\lambda,\hat\mu)$ returned by the certification layer defines a valid Benders optimality cut.
\end{prop}

\begin{proof}
Fix any $y$ for which the recourse LP \eqref{eq:opt_sp_lp} is feasible and bounded with optimal value $Q(y)$, and let $(\hat\lambda,\hat\mu)$ be any dual-feasible pair for \eqref{eq:opt_sp_lp_dual}.
By weak duality between \eqref{eq:opt_sp_lp} and \eqref{eq:opt_sp_lp_dual}, $d(y)^\top\hat\lambda + g^\top\hat\mu \le Q(y)$.
Using $d(y)=b-Ay$ and the definitions \eqref{eq:proxy_coeff_opt}, this is equivalent to $Q(y)\ge \alpha + \beta^\top y$, so $\theta \ge \alpha + \beta^\top y$ is a valid Benders optimality cut.
\end{proof}

Proposition~\ref{prop:proxy_cut_valid} holds for \emph{any} network parameters $\theta$, i.e., before, during, and after training. The projection--completion pipeline is the structural mechanism that converts arbitrary neural outputs into provably valid cuts. Training controls only how \emph{strong} those cuts are. With validity established, it is now possible to characterize the termination behavior of the online Proxy-BD loop (Section~\ref{sec:proxy_online}).

\begin{defi}
\label{def:proxy_fixedpoint}
A master solution $(\bar y^\star, \bar\theta^\star)$ is called a \emph{proxy fixed point} if no proxy-generated cut is violated at $(\bar y^\star, \bar\theta^\star)$.
\end{defi}

Since the proxy only approximates the oracle, a proxy fixed point is generally \emph{not} a global optimum of the original problem. The following remark clarifies the relationship between conditional completion optimality and global dual optimality.

\begin{rema}
\label{rem:proxy_conditional_opt}
Theorem~\ref{thm:opt_strongest_completion_lp} guarantees that, for a \emph{fixed} projected $\hat\lambda$, the completion $\hat\mu^\star(\hat\lambda)$ yields the strongest dual-feasible completion associated with $\hat\lambda$. However, this optimality is conditional on the projected multiplier $\hat\lambda$, since the overall dual problem jointly optimizes over $(\lambda,\mu)$. The proxy cut matches an oracle optimality cut when $(\hat\lambda,\hat\mu)$ is an optimal dual solution of~\eqref{eq:opt_sp_lp_dual}; dual non-uniqueness means many such $\hat\lambda$ can yield the same cut. The gap between the proxy bound and $Q(\bar y)$ is precisely what the self-supervised training objective \eqref{eq:proxy_train_loss} seeks to reduce: increasing $L$ drives $(\hat\lambda,\hat\mu)$ toward stronger dual solutions, tightening the resulting cuts without requiring dual-solution labels.
\end{rema}

\noindent\textbf{Optional deployment refinement.} Because Proxy-BD relies on approximate separation, the master ranks integer incumbents using an \emph{underestimate} of their true two-stage cost since the proxy bound lower-bounds $Q$, and can terminate at an incumbent that is not the cheapest one it visited. When the exact recourse value $Q(\bar y)$ of an integer master solution is inexpensive to evaluate, Proxy-BD may re-price the integer incumbents generated during the search by their exact total cost $f^\top y+Q(y)$ and return the cheapest. This \emph{true-cost selection} ranks only solutions visited during the run. It adds no oracle call inside the search and leaves the cut machinery and validity guarantees unchanged, yet it recovers the cheapest fully evaluated design encountered during the search whenever the proxy ranks them in the wrong order. The refinement is optional and orthogonal to cut generation (see Section~\ref{sec:results_fls} for deployment in large-scale facility location and Table~\ref{tab:fls_ablation} for effect quantification).

\noindent \textbf{Proxy-BD Termination.}
Proxy-BD terminates after finitely many iterations under mild assumptions on the separator and deployment environment. In the outer-loop implementation, finite termination follows from the finiteness of the first-stage design space and deterministic proxy separation. In branch-and-Benders-cut deployments, Proposition~\ref{prop:term_bbc} shows that the proxy separator cannot generate infinitely many cuts under standard numerical separation assumptions. Combined with the finite termination of the underlying branch-and-bound search, this yields finite termination of the B\&BC deployment. The following results formalize these termination properties.
\begin{rema}[Finite termination of outer-loop Proxy-BD]
\label{rem:term_outer}
Suppose the first-stage set $Y$ is finite, every restricted master problem has a finite optimum and is solved to global optimality, and the proxy separator is deterministic, producing at most one distinct cut for each design $y\in Y$. Then the outer-loop Proxy-BD algorithm of Section~\ref{sec:proxy_online} adds at most $|Y|$ cuts and terminates after finitely many iterations at a proxy fixed point.
\end{rema}
The result follows immediately from the finiteness of $Y$ and the fact that at most one proxy cut can be generated for each design. Consequently, only finitely many cuts can ever be added.
\begin{prop}[Finite proxy-cut generation in B\&BC]
\label{prop:term_bbc}
Consider a B\&BC deployment in which every proxy cut is retained in the master relaxation. If the proxy is invoked only at integer incumbents $y\in Y$, then under the assumptions of Remark~\ref{rem:term_outer}, at most $|Y|$ proxy cuts are added. If the proxy is also invoked at fractional callback points, assume in addition that: \emph{(i)} a cut is added only when violated by at least $\epsilon_{\mathrm{sep}}>0$, \emph{(ii)} all callback points lie in a compact set $S$, and \emph{(iii)} all generated cuts have uniformly bounded coefficients. Then only finitely many additional fractional proxy cuts are added.
\end{prop}
\begin{proof}
The integer-incumbent case follows directly from the finiteness of $Y$ and Remark~\ref{rem:term_outer}. Consider now fractional callback points. Write the cut generated at callback point $z_t$ as $h_t(z)\le 0$, where $h_t$ is affine and $z=(y,\theta)$ for optimality cuts, while $z=y$ for feasibility cuts. By the uniform coefficient bound, there exists $L<\infty$ such that every generated cut is $L$-Lipschitz on $S$:
\[
|h_t(z)-h_t(z')|\le L\,\lVert z-z'\rVert,\qquad \forall z,z'\in S,\ \forall t.
\]
Let $z_t$ and $z_s$ be two callback points at which cuts are generated, with $s>t$. Since cut $t$ is added only if it is violated by at least $\epsilon_{\mathrm{sep}}$, then $h_t(z_t)\ge \epsilon_{\mathrm{sep}}$. Since every proxy cut is globally valid and retained in the master relaxation, every later callback point satisfies all previously generated cuts. Hence, $h_t(z_s)\le 0$. Therefore,
        \[
        \epsilon_{\mathrm{sep}}
        \;\le\; h_t(z_t)-h_t(z_s)
        \;\le\; \lvert h_t(z_t)-h_t(z_s)\rvert
        \;\le\; L\,\lVert z_t-z_s\rVert,
        \]
which gives $\lVert z_t-z_s\rVert \ge \epsilon_{\mathrm{sep}}/L$. Thus, any two callback points that generate fractional cuts are separated by at least $\epsilon_{\mathrm{sep}}/L$. A compact set cannot contain infinitely many points that are pairwise separated by a positive distance. Hence, only finitely many fractional proxy cuts can be added.
\end{proof}
Proposition~\ref{prop:term_bbc} establishes finite proxy-cut generation under standard numerical separation assumptions.\\
\noindent\textbf{Stopping criterion.} In the outer-loop implementation, Proxy-BD terminates when a proxy fixed point is reached. In B\&BC deployments, termination occurs when the branch-and-bound search completes and no further proxy cuts are generated by the callback mechanism. Additional stopping criteria such as time limits, iteration budgets, or, whenever available, primal-dual gap tolerances may also be used.

\section{Experimental Design}\label{section:experimental_design}

The experiments evaluate Proxy-BD on three benchmark families and investigate two main questions: whether the certification mechanism consistently produces valid Benders cuts independently of prediction quality, and whether amortized proxy separation delivers substantial computational gains while maintaining near-optimal solutions as the recourse becomes more expensive. The remainder of this section is organized as follows. Section~\ref{sec:impl_env} describes the implementation environment. Section~\ref{sec:instances} introduces the benchmark instances. Section~\ref{sec:methods_metrics} presents the solution methods and evaluation criteria.

\subsection{Implementation Environment}\label{sec:impl_env}
All experiments run on exclusive Intel Cascade Lake Gold~6226 compute nodes with 64\,GB of memory. Optimization uses Gurobi~12.0.3 in single-threaded mode (\texttt{Threads}~$=1$) with default presolve, cuts, heuristics, and tolerances and a fixed seed; only the parameters below change. Each held-out evaluation instance runs with \texttt{MIPGap}~$=10^{-4}$ and a one-hour wall-clock limit. Proxy training and held-out evaluation are separate Slurm jobs and use the same hardware class. \emph{Offline state generation} solves each training instance with the exact oracle at \texttt{MIPGap}~$=0$ under a generous limit and records its master, subproblem, dual, and cut sequence; since no reported runtime uses state-generation time, these runs use shared compute. \emph{Held-out evaluation} compares the \emph{exact oracle} against the \emph{proxy-only run}, which replaces \emph{every} subproblem solve with the learned predict--project--complete proxy and never falls back to an exact solve. The two share the same master and Benders scheme and differ only in how each cut is separated.

\subsection{Instances}\label{sec:instances}

\begin{table}[!t]
\centering
\small
\caption{Held-out instance dimensions, facility-location families.
$m$: clients; $n$: facilities; $N$: paired test instances.}
\label{tab:instances_fl}
\begin{tabular}{lrrrr}
\toprule
Family & Shape & $m$ & $n$ & $N$ \\
\midrule
CAP & $50\times16$     & 50   & 16   & 1200 \\
CAP & $50\times25$     & 50   & 25   & 1200 \\
CAP & $50\times50$     & 50   & 50   & 1200 \\
CAP & $1000\times100$  & 1000 & 100  & 1660 \\
\midrule
UFL & $200\times200$   & 200  & 200  & 100 \\
UFL & $300\times300$   & 300  & 300  & 100 \\
UFL & $500\times500$   & 500  & 500  & 100 \\
UFL & $1000\times1000$ & 1000 & 1000 & 20  \\
UFL & $2000\times2000$ & 2000 & 2000 & 20  \\
\bottomrule
\end{tabular}
\end{table}

\begin{table}[!t]
\centering
\small
\caption{Held-out instance dimensions, network design (MCNDP).
$|V|$: nodes; $|E|$: arcs; $|I|$: commodities; $N$: paired test instances.}
\label{tab:instances_mcndp}
\begin{tabular}{llrrrr}
\toprule
Class & Shape & $|V|$ & $|E|$ & $|I|$ & $N$ \\
\midrule
r03 & $10\times35$ & 10 & 35 & 50 & 120 \\
r06 & $10\times60$ & 10 & 60 & 50 & 120 \\
r09 & $10\times83$ & 10 & 83 & 50 & 120 \\
\bottomrule
\end{tabular}
\end{table}

\noindent \textbf{Instance Overview.} The benchmark draws on three families chosen for complementary recourse
structures. The first family is the capacitated facility location problem (CAP) drawn from the OR-Library \citep{beasley1988algorithm}. CAP instances have a cheap transportation recourse. The second family is the uncapacitated facility location problem (UFL) drawn from \href{https://resources.mpi-inf.mpg.de/departments/d1/projects/benchmarks/UflLib/packages.html}{UflLib}. UFL instances have a closed-form per-client recourse that scales to $2000\times2000$. The third family is the multicommodity capacitated network design problem (MCNDP) drawn from the $R$ set \citep{crainic2001bundle}. MCNDP instances have a coupled, capacitated routing recourse that can be infeasible.

\noindent \textbf{Instance Generation.} For each base instance from the CAP, UFL, and MCNDP, perturbed variants are generated by independent lognormal noise on the cost and capacity data, then split $50/25/25$ into train/validation/test by a fixed seed with no base instance shared across partitions. Tables~\ref{tab:instances_fl} and~\ref{tab:instances_mcndp} report the held-out instance dimensions per shape: the number of clients $m$ and facilities $n$ for the facility-location families, and the number of nodes $|V|$, arcs $|E|$, and commodities $|I|$ for MCNDP. The paired test count $N$ is the last column of each table. The three families span two orders of magnitude in instance dimension. CAP ranges from $50\times16$ to $1000\times100$, with the largest CAP shape evaluating master-bound behavior at instance dimensions comparable to small UFL. UFL ranges from $200\times200$ to $2000\times2000$, with the two largest shapes used to stress the case in which the recourse, not the master search, drives the wall-clock cost. MCNDP fixes the node count at $10$, and the commodity count at $50$, while expanding arc density from $35$ to $83$, thereby increasing the coupling of the capacitated routing recourse. The differences across families in recourse structure, not in raw dimension alone, determine where the proxy helps most. Training and validation perturbations drive offline state generation. The test perturbations are reserved for held-out evaluation.

\subsection{Solution Methods and Evaluation Criteria}\label{sec:methods_metrics}

\begin{table}[!t]
\centering
\small
\caption{Exact Benders oracle per family with $w$ being the in-out stabilization mixing weight.}
\label{tab:oracles}
\adjustbox{max width=\textwidth}{%
\begin{tabular}{lllcl}
\toprule
Family & Benders scheme & Cut families & $w$ & Master side constraint \\
\midrule
CAP   & Outer-loop  & Optimality              & $0.5$ & $\sum_j s_j y_j \ge \sum_i d_i$ \\
UFL   & B\&BC \citep{fischetti2017redesigning} & Optimality & --- & $\sum_j y_j \ge 1$ \\
MCNDP & B\&BC       & Optimality \& Feasibility  & $0.1$ & --- \\
\bottomrule
\end{tabular}}
\end{table}

The families differ in how the proxy is invoked. CAPs are solved using a classical Benders decomposition, with one proxy call generating a cut (see Figure~\ref{fig:proxy_online}). UFL and MCNDP use B\&BC \citep{fischetti2017redesigning,moreno2019branch,satici2026branch} (see Figure~\ref{fig:bbc_loop}). In both, the proxy is the same predict--project--and--complete pipeline of Section~\ref{sec:proxy_unified}.

\noindent \textbf{Exact Oracles and Baselines.} The baseline is a fully configured exact Benders oracle. It both generates the offline training states and provides the performance benchmark. Each oracle separates cuts from exact subproblem duals, stabilizes the separation point by in-out stabilization, and uses Gurobi~12.0.3 with lazy callbacks and warm-started dual reuse. The three oracles differ only in the recourse structure and the resulting cut (Table~\ref{tab:oracles}). For CAP, each subproblem is a transportation LP whose dual yields a knapsack-lifted optimality cut. For UFL, the recourse decomposes by client in closed form, $Q(y)=\sum_i\min_{j:\,y_j=1}c_{ij}$. The state-of-the-art (SOTA) scheme of \citet{fischetti2017redesigning} is used as the UFL baseline. For MCNDP, the coupled multicommodity-flow recourse can be infeasible at a candidate design. In such a case, the oracle emits a knapsack-lifted optimality cut at feasible incumbents and a slice-normalized Farkas feasibility cut at infeasible ones (see Appendix~\ref{B} for more details). At each separation~$t$, the oracle records a Benders state $(\xi,\hat y^{(t)})$, with $\hat y^{(t)}$ the stabilized separation point and $\xi$ the instance data. The states accumulated over all training instances form the dataset $\mathcal S$.

\noindent \textbf{Proxy Models, Repair, and Training.} Each family has its own proxy, but all three share the same certification layer of Section~\ref{sec:proxy_unified}, and it is this shared layer, not per-family tuning, that makes the cuts valid. Because the projection and completion return a dual-feasible certificate for \emph{any} network output, every emitted cut is a valid Benders cut (Proposition~\ref{prop:proxy_cut_valid}). The experiments, therefore, measure only what \emph{varies}, including cut strength, speed, and, for MCNDP, feasibility restoration. Only the \emph{predict} step differs (Table~\ref{tab:proxies}). CAP uses a full-state MLP whose input concatenates the instance data and separation point $\hat y$ (dimension $3n+m+mn$), with a Softplus output absorbing the nonnegativity projection. UFL uses a \emph{row-wise} predictor. The multiplier for client $i$ depends only on that client's cost row and $\hat y$, so one network is shared across clients (Appendix~\ref{B}). MCNDP uses a structured MLP for the per-(commodity, node) flow-conservation multipliers $\hat u$ where arcs are ordered by cost for index-invariance. A one-line \emph{penalty projection} enforcing $\hat u_{i,o_i}-\hat u_{i,t_i}\le P\,d_i$ keeps the induced optimality cut valid, after which the same continuous-knapsack completion as CAP lifts it to the binary design (Appendix~\ref{B}). A separate model is trained per instance shape, since feature and output dimensions scale with the instance.

The UFL row-wise proxy is deployed within the framework of \citet{fischetti2017redesigning} and trained on the states generated by that scheme. During deployment, two rules, which do not alter how the cuts are derived, are applied. First, the per-client cuts are aggregated into a single optimality cut for each integer incumbent. This ensures the master problem grows by only one cut per incumbent, rather than one per client. Second, the proxy uses its predictions directly. To choose the final solution, a true-cost selection rule is applied. This rule ranks all the integer designs visited during the run and returns the one with the lowest exact two-stage cost, $f^\top y+Q(y)$. Because the subproblem cost $Q(y)=\sum_i\min_{j:\,y_j=1}c_{ij}$ can be evaluated in closed form, this final selection simply evaluates known designs and requires no additional oracle calls.

For the MCNDP, if the proxy's predicted multiplier under-prices the recourse, the resulting slice-Farkas feasibility cut at $\bar{y}$ might not separate an infeasible design. As a result, the master problem might converge on a design that lacks the capacity to route all demand. Note that this is an issue of primal feasibility, not cut validity.To restore feasibility, we use a LP-free repair process. First, a primal-flow certificate checks whether the currently open arcs can route every commodity. If they cannot, we greedily open additional arcs and rank them by the proxy's predicted dual prices, and recheck the certificate. This loop repeats until all demand is successfully routed (a process guaranteed to terminate, since opening all arcs is always feasible). Finally, we perform a single exact solve on the repaired design to measure its true gap.

\noindent\textbf{Training.} All proxies use the self-supervised objective~\eqref{eq:proxy_train_loss} with Adam (learning rate $10^{-3}$, plateau decay to a $10^{-5}$ floor, early stopping after eight non-improving validation checks, best checkpoint retained) on a fixed optimizer-step budget (Table~\ref{tab:proxies}). MCNDP batches mix optimality and feasibility states. The optimality head is trained on feasible-incumbent states by maximizing the certified dual bound, and the feasibility head is trained on infeasible-incumbent states by maximizing the slice-normalized Farkas separation value~\eqref{eq:phase1_lp_dual}, both with the multipliers passed through their respective projection step before the cut is formed (Appendix~\ref{B}).

\begin{table}[!t]
\centering
\small
\caption{Proxy model and training budget per family (Adam, lr $10^{-3}$, plateau decay, early stopping).}
\label{tab:proxies}
\begin{tabular}{lllccc}
\toprule
Family & Proxy model & Hidden units & Batch & Step budget & Val.\ every \\
\midrule
CAP   & Full-state MLP & $512,512$ & $512$ & $2^{20}$ & $2048$ \\
UFL   & Row-wise MLP   & $512,512$ & $512$ & $2^{20}$ & $2048$ \\
MCNDP & Structured MLP & $256,256$ & $256$ & $2^{17}$ & $4096$ \\
\bottomrule
\end{tabular}
\end{table}

\noindent \textbf{Evaluation Metrics.}
The following metrics are reported. The \emph{true optimality gap} of a design $\bar y$ is $(f^\top\bar y+Q(\bar y)-z^\star)/z^\star$, where $Q(\bar y)$ is the exact recourse value from a solve performed \emph{after} the run (outside the timed comparison), $f^\top\bar y+Q(\bar y)$ the resulting exact two-stage cost, and $z^\star$ the exact optimum. When the oracle hits the one-hour limit, and no exact optimum is available, the \emph{gap to the oracle's one-hour design} is reported, and mark such entries. The true gap on the Oracle-certified subset is reported alongside. For MCNDP, the post-repair \emph{feasibility rate} and the repair's activation frequency and cost are also reported. The \emph{speedup} is the per-instance oracle/proxy wall-clock ratio, reported as median and mean. The \emph{warm} speedup (model-load time amortized) is used throughout. Under one-hour censoring, note that the per-instance median can differ from the ratio of median times. For the accelerated UFL deployment, the number of cuts added is reported, a hardware-independent measure of solver effort.

\section{Computational Results}\label{sec:computational_results}

This section reports the held-out computational results across the three benchmark families. The study is organized into two parts. The first evaluates Proxy-BD in the classical Benders setting to analyze cut quality, certification behavior, and approximation properties relative to exact separation. The second evaluates Proxy-BD within modern large-scale decomposition schemes, where repeated node-level separation makes amortized proxy inference particularly effective computationally.

\subsection{Summary of Results}\label{sec:results_summary}

Table~\ref{tab:summary} summarizes the held-out computational performance across all benchmark families, highlighting the relationship between recourse complexity, approximation quality, and amortized speedup. In the cheap-recourse CAP regime, Proxy-BD closely matches the exact oracle with sub-$1\%$ median true gaps, but acceleration remains limited because master re-optimization dominates the runtime. As the recourse becomes more computationally expensive, the proxy's speed advantage grows. On the large-scale UFL family, where repeated separation is the dominant bottleneck, Proxy-BD achieves up to $161\times$ speedups while maintaining near-optimal solutions with median gaps below $0.5\%$. MCNDP falls in between: the coupled network recourse is substantially harder than CAP and benefits from proxy acceleration, but the stronger coupling structure also makes cut approximation more challenging, producing a wider gap range while still preserving validity and feasibility throughout all held-out runs. Overall, the summary confirms the central empirical finding of the paper: the computational benefit of Proxy-BD increases with recourse complexity, while the certification layer keeps the cuts valid regardless of prediction quality.

\begin{table}[!t]
\centering
\small
\caption{Held-out summary: median true optimality gap of the proxy against each family's exact oracle, and peak per-instance speedup. The recourse is feasible at every visited design, and the optimality gap is taken against the exact optimum on every test instance.}
\label{tab:summary}
\adjustbox{max width=\textwidth}{%
\begin{tabular}{lrrrrl}
\toprule
Family & Test shapes & Test $N$ & Median true gap (\%) & Peak speedup & Baseline \\
\midrule
CAP & $50{\times}16$--$1000{\times}100$     & $1200$--$1660$ & $0.00$--$0.58$ & $3.6\times$ & Outer-loop \\
UFL & $200{\times}200$--$2000{\times}2000$  & $20$--$100$    & $0.00$--$0.50$ & $161\times$ & SOTA B\&BC \\
MCNDP & $10{\times}35$--$10{\times}83$  & $120$    & $0.012$--$2.65$ & $27.05\times$ & B\&BC \\
\bottomrule
\end{tabular}}
\end{table}

\subsection{CAP Instances}\label{sec:results_cap}

CAP has a cheap recourse: each transportation subproblem is inexpensive relative to the master problem, so the proxy has the least opportunity to reduce solve time. The family, therefore, serves primarily as a calibration benchmark: whether Proxy-BD can closely mimic exact separation when there is little computational burden to amortize.

On the three small shapes in Table~\ref{tab:cap_results}, the proxy attains a median true gap of $0.39$--$0.58\%$ at median speedups of $2.4$--$3.6\times$; the means exceed the medians (e.g.\ $4.43\%$ gap and $28.7\times$ speedup on $50\times25$) because the distributions are right-skewed: a small subset of instances is simultaneously harder for the proxy and more expensive for the oracle. Both statistics are therefore reported.

Figure~\ref{fig:cap_gap} gives the full distribution. On the largest CAP shape, $1000\times100$, over $1660$ held-out instances the proxy is exact at the median ($0.00\%$ median, $0.11\%$ mean true gap) yet does \emph{not} accelerate (slower at the median, $0.64\times$, though faster on average, $1.83\times$), because in outer-loop Benders at this size the repeated master re-solve, not the cheap transportation subproblem, is the bottleneck, so replacing the subproblem yields no net speedup. Solution quality is therefore preserved even at the largest CAP scale. Speedup, however, is not. The two large-scale families examined next have far more expensive recourse, where the proxy's advantage is largest.

\begin{table}[!t]
\centering
\footnotesize
\setlength{\tabcolsep}{4pt}
\caption{Held-out proxy vs.\ the exact outer-loop oracle on CAP. \emph{Gap} is the true optimality gap. \emph{Time} reports both median and mean wall-clock per instance. \emph{Speedup} the per-instance oracle/proxy ratio (median and mean). \emph{Cuts} the per-instance optimality cut count (median and mean). The oracle column is listed first throughout. Bold marks $<1.5\%$ gap, $>1\times$ speedup, and the proxy cut count.}
\label{tab:cap_results}
\begin{tabular}{lrrrrrrrrrrrrr}
\toprule
 & & \multicolumn{2}{c}{Gap (\%)} & \multicolumn{2}{c}{Oracle time (s)} & \multicolumn{2}{c}{Proxy time (s)} & \multicolumn{2}{c}{Speedup} & \multicolumn{2}{c}{Oracle cuts} & \multicolumn{2}{c}{Proxy cuts} \\
\cmidrule(lr){3-4}\cmidrule(lr){5-6}\cmidrule(lr){7-8}\cmidrule(lr){9-10}\cmidrule(lr){11-12}\cmidrule(lr){13-14}
Shape & $N$ & med & mean & med & mean & med & mean & med & mean & med & mean & med & mean \\
\midrule
$50\times16$    & 1200 & \textbf{0.44} & 1.65          & 0.22   & 0.42   & 0.08   & 0.18   & $\mathbf{2.4\times}$ & $\mathbf{8.5\times}$  & 18  & 21  & \textbf{11}  & \textbf{13} \\
$50\times25$    & 1200 & \textbf{0.58} & 4.43          & 2.56   & 3.74   & 1.00   & 2.04   & $\mathbf{2.4\times}$ & $\mathbf{28.7\times}$ & 44  & 50  & \textbf{27}  & \textbf{32} \\
$50\times50$    & 1200 & \textbf{0.39} & 0.92          & 6.79   & 28.32  & 2.16   & 3.93   & $\mathbf{3.6\times}$ & $\mathbf{36.5\times}$ & 42  & 68  & \textbf{27}  & \textbf{31} \\
$1000\times100$ & 1800 & \textbf{0.00} & \textbf{0.11} & 149.60 & 386.71 & 410.93 & 876.72 & $0.6\times$          & $\mathbf{1.8\times}$  & \textbf{87}  & \textbf{102} & 148 & 144 \\
\bottomrule
\end{tabular}
\end{table}

\begin{figure}[!t]
\centering
\includegraphics[width=0.85\textwidth]{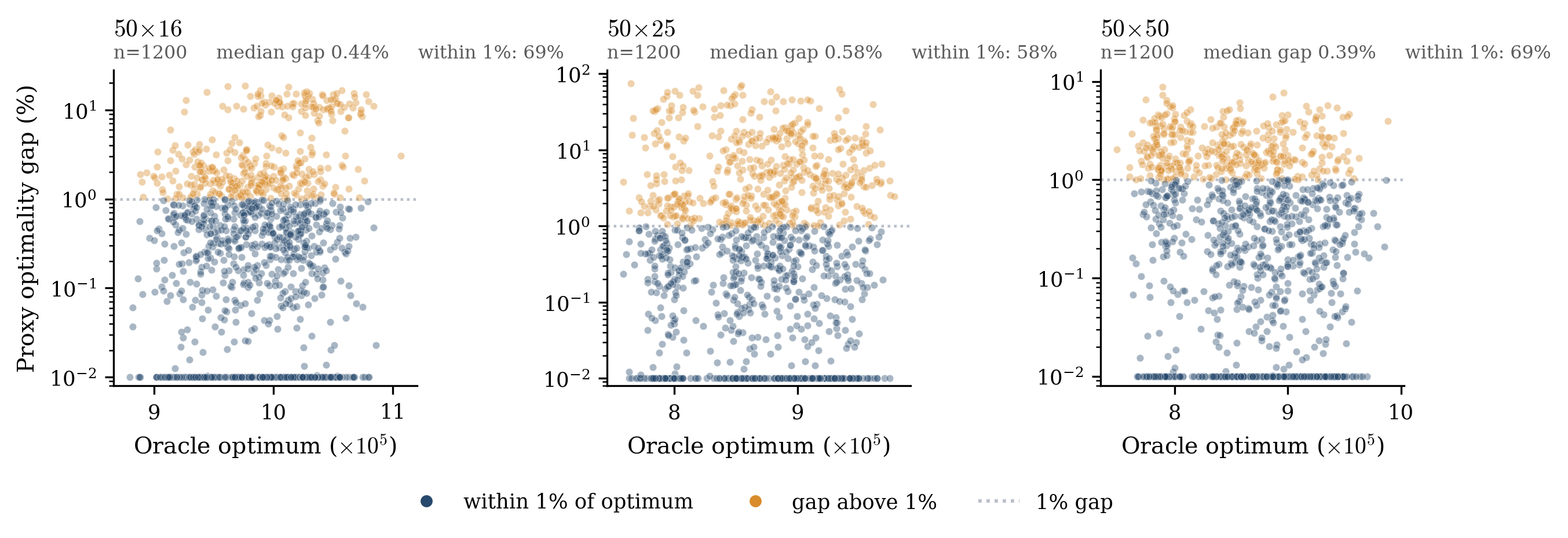}
\caption{Per-shape distribution of the proxy's true optimality gap on held-out
CAP instances, $(Q(\bar y)-Q^\star)/Q^\star$ with the recourse evaluated exactly.}
\label{fig:cap_gap}
\end{figure}

\subsection{UFL Instances}\label{sec:results_fls}

UFL is the family on which the recourse, not the master search, dominates the wall-clock cost, and is where the proxy's advantage is largest. The UFL is deployed within the SOTA B\&BC scheme of \citet{fischetti2017redesigning}. Table~\ref{tab:headline_fls} reports the comparison and Figure~\ref{fig:fls_heldout} the per-shape distributions. On every shape from $200\times200$ to $2000\times2000$, the accelerated proxy returns designs with mean true gaps below 0.7\%, and at the exact optimum for at least half the instances on the two smallest shapes: median gap $0.00\%$, rising to $0.50\%$ at $500\times500$. The $200\times200$ and $300\times300$ rows serve as calibration: accuracy is preserved while the speedups are modest because the accelerated oracle is itself inexpensive at these sizes. The key entries are the two largest shapes, left at $\sim70\%$ gap under the classic B\&BC scheme (see Appendix \ref{C}). The speedup broadly increases with the cost of exact node-level separation, reaching $161\times$ median at $2000\times2000$ where the oracle hits the one-hour limit, and the proxy attains it while adding one to two orders of magnitude fewer cuts: a median of $115$ aggregated cuts versus
$27{,}819$ for the oracle at $2000\times2000$.

\begin{table}[!t]
\centering
\footnotesize
\setlength{\tabcolsep}{4pt}
\caption{Held-out accelerated proxy (true-cost selection) vs.\ the exact SOTA oracle on UFL. For Time, Speedup, and Cuts, both median and mean across the $N$ matched instances are reported. Bold marks $<1.5\%$ gap, $>1\times$ speedup, and the proxy cut count.}
\label{tab:headline_fls}
\begin{tabular}{lrrrrrrrrrrrrr}
\toprule
 & & \multicolumn{2}{c}{Gap (\%)} & \multicolumn{2}{c}{Oracle time (s)} & \multicolumn{2}{c}{Proxy time (s)} & \multicolumn{2}{c}{Speedup} & \multicolumn{2}{c}{Oracle cuts} & \multicolumn{2}{c}{Proxy cuts} \\
\cmidrule(lr){3-4}\cmidrule(lr){5-6}\cmidrule(lr){7-8}\cmidrule(lr){9-10}\cmidrule(lr){11-12}\cmidrule(lr){13-14}
Shape & $N$ & med & mean & med & mean & med & mean & med & mean & med & mean & med & mean \\
\midrule
\quad $200\times200$   & 100 & \textbf{0.00} & 0.32 & 3.11             & 3.26             & 1.30   & 4.70   & $\mathbf{1.8\times}$   & $\mathbf{5.5\times}$   & 2797  & 2825  & \textbf{122} & \textbf{184} \\
\quad $300\times300$   & 100 & \textbf{0.00} & 0.16 & 5.67             & 8.43             & 2.69   & 8.40   & $\mathbf{2.2\times}$   & $\mathbf{3.0\times}$   & 4147  & 4196  & \textbf{170} & \textbf{218} \\
\quad $500\times500$   & 100 & \textbf{0.50} & 0.61 & 98.8             & 123.4            & 3.52   & 9.13   & $\mathbf{20.7\times}$  & $\mathbf{34.1\times}$  & 7171  & 7164  & \textbf{133} & \textbf{180} \\
\quad $1000\times1000$ &  20 & \textbf{0.36} & 0.57 & 484.9            & 440.7            & 35.1   & 45.7   & $\mathbf{13.5\times}$  & $\mathbf{47.3\times}$  & 13467 & 13470 & \textbf{362} & \textbf{346} \\
\quad $2000\times2000$ &  20 & \textbf{0.48} & 0.49 & $3600^{\dagger}$ & $3600^{\dagger}$ & 22.6   & 24.0   & $\mathbf{161\times}$   & $\mathbf{183\times}$   & 27819 & 27839 & \textbf{115} & \textbf{123} \\
\bottomrule
\end{tabular}

\vspace{3pt}
{\footnotesize $^{\dagger}$\, The accelerated oracle reached the one-hour limit on a majority of $2000\times2000$ instances; the gap is then relative to the oracle's best incumbent, and the speedup is a lower bound.}
\end{table}

\begin{figure}[!t]
\centering
\includegraphics[width=0.85\textwidth]{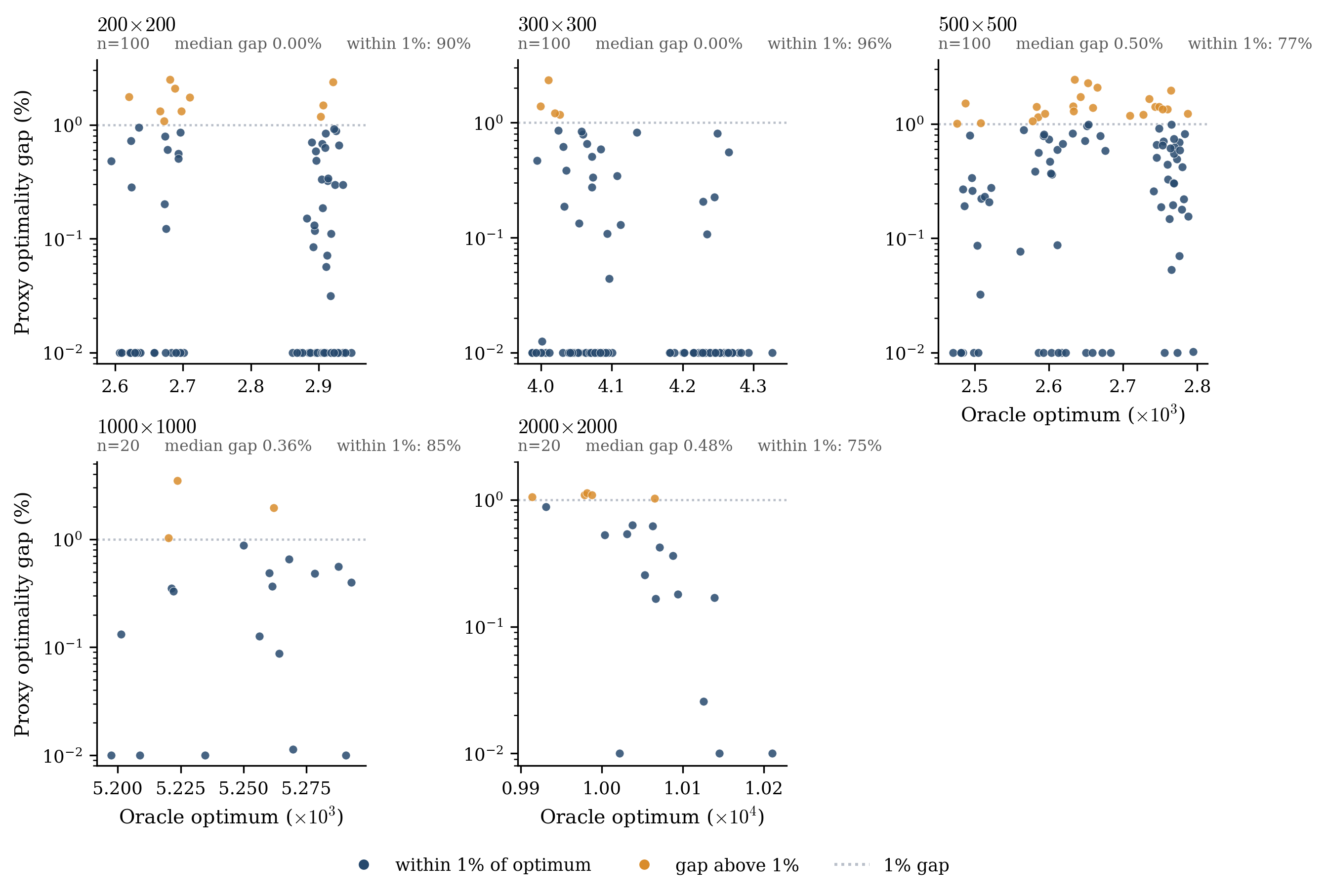}\\[4pt]
\includegraphics[width=0.85\textwidth]{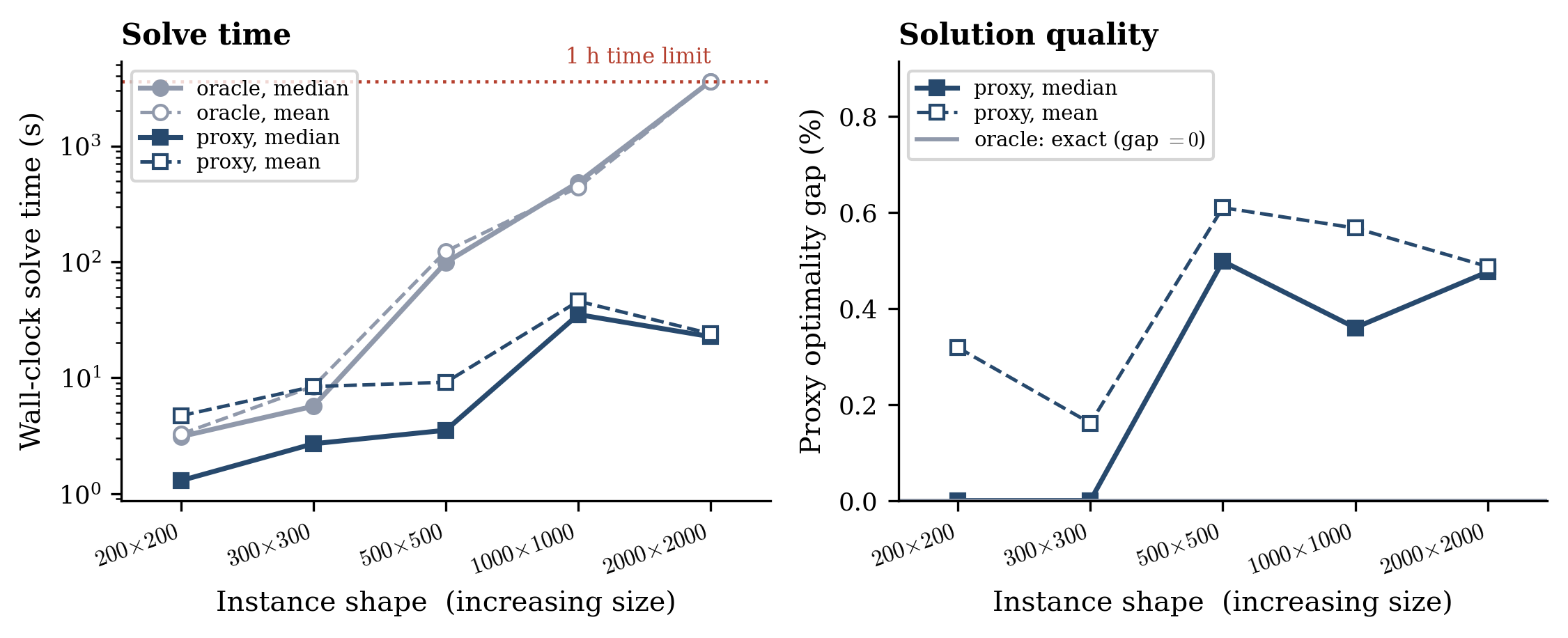}
\caption{Held-out SOTA proxy on UFL. Top: per-shape distribution of the proxy's true optimality gap, within $1\%$ of the optimum in navy. Bottom: scaling with instance size of solve time (left; the accelerated oracle reaches the one-hour limit at $2000\times2000$) and of the proxy's true gap (right), which stays below $0.7\%$ on every shape.}
\label{fig:fls_heldout}
\end{figure}

The benefit of true-cost selection on the same runs is measured by comparing the master's final incumbent against the design chosen by the selection rule (see Table~\ref{tab:fls_ablation}). If the master's final incumbent is solely relied on, essentially deploying the aggregated cuts without the selection step, the median true gap remains high, between $5.9\%$ and $16.6\%$. However, applying true-cost selection to those identical runs brings the gap down to $0.0$--$0.5\%$. The reason is that the branch-and-bound process does discover near-optimal solutions during the search, but the approximate cuts misprice them, causing the solver to rank them incorrectly. By re-pricing all the visited integer solutions using their exact cost, this error is corrected, and the true best design is recovered. This exact re-pricing uses the closed-form $Q(y)=\sum_i\min_{j:\,y_j=1}c_{ij}$ (Appendix~\ref{B}).
\begin{table}[!t]
\centering
\footnotesize
\setlength{\tabcolsep}{4pt}
\caption{Ablation of true-cost selection on the SOTA proxy. \emph{Without}: return the master's terminal incumbent. \emph{With}: return the visited integer master solution of the least exact cost. Selection is a closed-form re-pricing of visited solutions, so the per-instance proxy time and cut count are identical under the two treatments and are reported as shared columns. Bold marks the better gap column.}
\label{tab:fls_ablation}
\begin{tabular}{lrrrrrrrrr}
\toprule
 & & \multicolumn{2}{c}{Without selection (\%)} & \multicolumn{2}{c}{With selection (\%)} & \multicolumn{2}{c}{Proxy time (s)} & \multicolumn{2}{c}{Proxy cuts} \\
\cmidrule(lr){3-4}\cmidrule(lr){5-6}\cmidrule(lr){7-8}\cmidrule(lr){9-10}
Shape & $N$ & med & mean & med & mean & med & mean & med & mean \\
\midrule
\quad $200\times200$   & 100 & 9.56  & 11.40 & \textbf{0.00} & \textbf{0.32} & 1.30  & 4.70  & 122 & 184 \\
\quad $300\times300$   & 100 & 6.74  & 7.09  & \textbf{0.00} & \textbf{0.16} & 2.69  & 8.40  & 170 & 218 \\
\quad $500\times500$   & 100 & 7.47  & 8.19  & \textbf{0.50} & \textbf{0.61} & 3.52  & 9.13  & 133 & 180 \\
\quad $1000\times1000$ &  20 & 16.56 & 15.80 & \textbf{0.36} & \textbf{0.57} & 35.07 & 45.73 & 362 & 346 \\
\quad $2000\times2000$ &  20 & 5.93  & 7.42  & \textbf{0.48} & \textbf{0.49} & 22.60 & 23.97 & 115 & 123 \\
\bottomrule
\end{tabular}
\end{table}

Figure~\ref{fig:fls_convergence} renders the same comparison as a trajectory. Because true-cost selection returns the visited integer design of least exact cost, the accelerated proxy's best-incumbent true gap is monotone over the search and terminates exactly at the returned design's gap, unlike the single-shot proxy of Figure~\ref{fig:convergence}, reported as one marker. On every shape, the proxy reaches a near-optimal design while separating one to two orders of magnitude fewer cuts than the accelerated oracle needs to close its provable gap.

\begin{figure}[!t]
\centering
\includegraphics[width=0.85\textwidth]{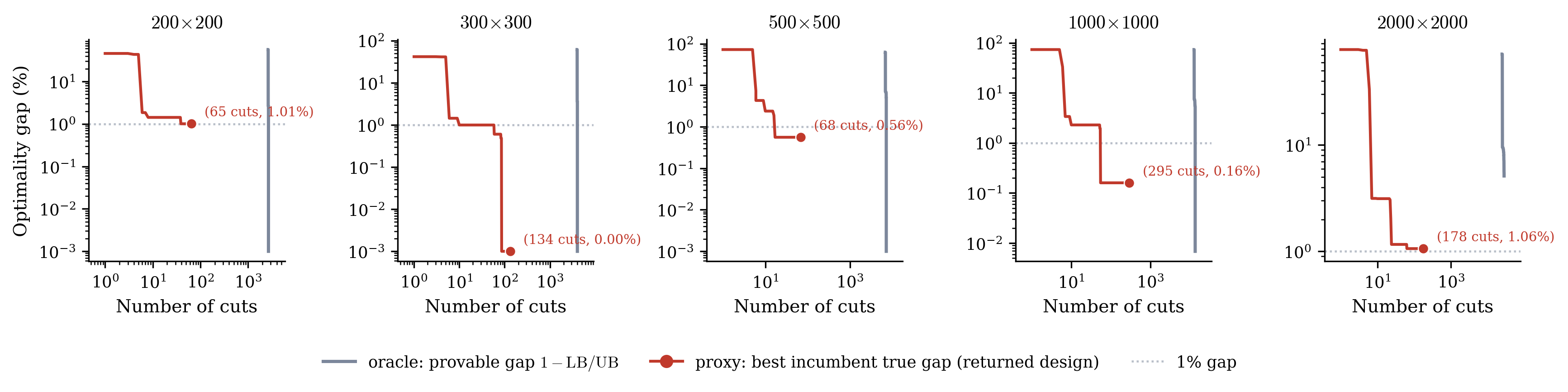}
\caption{Convergence under the SOTA deployment on a representative (median-speedup) instance per UFL shape, against the number of cuts separated. The accelerated oracle traces its provable gap $1-\mathrm{LB}/\mathrm{UB}$. The accelerated proxy traces the true gap of its best incumbent under true-cost selection (Section~\ref{sec:results_fls}), which is monotone and terminates at the returned design's gap. The proxy reaches a near-optimal design at one to two orders of magnitude fewer cuts.}
\label{fig:fls_convergence}
\end{figure}

\subsection{MCNDP Instances}\label{sec:results_mcndp}

MCNDP adds the harder test: a coupled, capacitated recourse that can be infeasible, handled at separation time by the slice-normalized Farkas feasibility cut of Section~\ref{sec:feas_reduce_to_opt}, with primal feasibility restored at deployment by the LP-free repair. On all $360$ held-out instances, the repair raises post-repair feasibility to $100\%$ on every shape, activating on $24.2\%$ of $10\times35$ instances (median one arc, $0.09\,\mathrm{s}$) up to $75.8\%$ of $10\times83$ instances (median four arcs, $3.2\,\mathrm{s}$).

The three shapes trace the same scaling pattern as UFL. On the smallest, $10\times35$, the coupled recourse LP solves in milliseconds, so the proxy stays close to the oracle runtime at $1.42\times$ median speedup with a true median gap of $0.012\%$ certified against the exact optimum on all $120$ instances. On the two largest shapes, the recourse is costly: the exact oracle hits the one-hour limit on $66.7\%$ of $10\times60$ and $66.7\%$ of $10\times83$ instances, certifying optimality on the other third where the proxy's true median gap is $0.567\%$ ($n=40$) and $2.65\%$ ($n=40$) as per Table~\ref{tab:summary_mcndp}. The deployment refinement here is master-side cut-set separation: at every integer master incumbent, the most-violated cut-set inequalities of the form $\sum_{e\in\delta^+(S)}u_e y_e\ge \sum_{i:o_i\in S,t_i\notin S}d_i$ are added as valid master cuts (proxy-independent, vectorized at instance load). This strengthens the master without touching the proxy and yields a speedup that grows with arc count: $1.42\times$, $9.37\times$, $27.05\times$ as the shape scales from $10\times35$ to $10\times83$. The largest gaps fall on the instances that the exact oracle also cannot certify within one hour. Table~\ref{tab:mcndp_results} reports the gap and speedup.

\begin{table}[!t]
\centering
\small
\caption{Held-out summary on MCNDP (slice-Farkas B\&BC with LP-free repair and master-side cut-set separation at integer master incumbents): the exact oracle reaches the $1\,\mathrm{h}$ limit on the two larger shapes, so the all-instances gap is reported against the oracle's $1\,\mathrm{h}$ design and the true optimality gap is reported only on the oracle-certified subset.}
\label{tab:summary_mcndp}
\adjustbox{max width=\textwidth}{%
\begin{tabular}{lrrrr}
\toprule
Shape & Test $N$ & Median gap vs.\ $1\,\mathrm{h}$ design (\%) & Median true gap (\%, certified) & Peak speedup \\
\midrule
$10{\times}35$ & $120$ & $0.012$               & $0.012$ ($n{=}120$) &  $1.42\times$ \\
$10{\times}60$ & $120$ & $3.97^{\ddagger}$   & $0.567$ ($n{=}40$)  & $9.37\times$ \\
$10{\times}83$ & $120$ & $7.42^{\ddagger}$   & $2.65$ ($n{=}40$)   & $27.05\times$ \\
\bottomrule
\end{tabular}}

\vspace{2pt}
{\footnotesize $^{\ddagger}$\,The exact oracle reached the $1\,\mathrm{h}$
limit on $66.7\%$ ($10{\times}60$) and $66.7\%$ ($10{\times}83$) of
instances; on the censored complement, the true optimum is unknown, and the gap is reported against the oracle's $1\,\mathrm{h}$ design.}
\end{table}

\begin{table}[!t]
\centering
\footnotesize
\setlength{\tabcolsep}{4pt}
\caption{Held-out proxy vs.\ exact oracle on MCNDP under the cut-set deployment. \emph{Gap} is the true optimality gap on $10\times35$ and the gap to the oracle's one-hour design on the larger shapes; proxy time includes the LP-free repair. Bold marks $<1.5\%$ gap and $>1\times$ speedup.}
\label{tab:mcndp_results}
\begin{tabular}{lrrrrrrrrrrrrr}
\toprule
 & & \multicolumn{2}{c}{Gap (\%)} & \multicolumn{2}{c}{Oracle time (s)} & \multicolumn{2}{c}{Proxy time (s)} & \multicolumn{2}{c}{Speedup} & \multicolumn{2}{c}{Oracle cuts} & \multicolumn{2}{c}{Proxy cuts} \\
\cmidrule(lr){3-4}\cmidrule(lr){5-6}\cmidrule(lr){7-8}\cmidrule(lr){9-10}\cmidrule(lr){11-12}\cmidrule(lr){13-14}
Shape & $N$ & med & mean & med & mean & med & mean & med & mean & med & mean & med & mean \\
\midrule
$10\times35$ & 120 & \textbf{0.012}    & \textbf{0.450} & 2.41             & 2.95             & 1.75 & 2.23  & $\mathbf{1.42\times}$  & $\mathbf{1.42\times}$  & 74    & 87   & 70   & 81   \\
$10\times60$ & 120 & $3.97^{\dagger}$  & 5.20           & $3600^{\dagger}$ & $2611^{\dagger}$ & 119  & 403   & $\mathbf{9.37\times}$  & $\mathbf{32.04\times}$ & 9053  & 8548 & 1984 & 2210 \\
$10\times83$ & 120 & $7.42^{\dagger}$  & 8.39           & $3600^{\dagger}$ & $2567^{\dagger}$ & 73   & 102   & $\mathbf{27.05\times}$ & $\mathbf{33.84\times}$ & 11273 & 8458 & 1132 & 1185 \\
\bottomrule
\end{tabular}

\vspace{3pt}
{\footnotesize $^{\dagger}$\, The oracle hits the one-hour limit on most instances of the shape; the gap is then against its best incumbent, and the speedup is a lower bound. On the oracle-certified subset, the true median gap is $0.567\%$ ($10\times60$, $n=40$) and $2.65\%$ ($10\times83$, $n=40$).}
\end{table}

\begin{table}[!t]
\centering
\small
\caption{Paired MCNDP upper-bound comparison under the deployment of Table~\ref{tab:mcndp_results}. $P{<}O$, $P{=}O$, $P{>}O$ count instances where the proxy UB is below, equal to (within $10^{-6}$), or above the oracle UB; $\Delta\mathrm{UB}=100(\mathrm{UB}_P-\mathrm{UB}_O)/\mathrm{UB}_O$. Paired $N$ excludes instances where the oracle returns no feasible UB at the one-hour limit.}
\label{tab:mcndp_ub_comparison}
\begin{tabular}{lrrrrr}
\toprule
Shape & Paired $N$ & $P<O$ & $P=O$ & $P>O$ & Mean $\Delta$ UB (\%) \\
\midrule
$10\times35$ & 113 & 0 & 56 & 57 & $+0.461$ \\
$10\times60$ & 110 & 1 & 9 & 100 & $+5.768$ \\
$10\times83$ & 108 & 0 & 1 & 107 & $+9.697$ \\
\bottomrule
\end{tabular}
\end{table}

Table~\ref{tab:mcndp_ub_comparison} reframes the MCNDP comparison at the returned-incumbent level for paired proxy and oracle runs, both under the cut-set deployment. The smallest shape is nearly tied: $56$ of $113$ pairs match the oracle upper bound, and the mean proxy increase is only $+0.461\%$. On $10\times60$ and $10\times83$, the proxy upper bound is usually above the oracle upper bound, with mean increases of $+5.768\%$ and $+9.697\%$, respectively, although one $10\times60$ pair has a strictly better proxy incumbent. This is consistent with the one-hour oracle censoring in Table~\ref{tab:mcndp_results}: the large-shape comparison is against the oracle's returned incumbent rather than a uniformly certified optimum.

\subsection{Mechanisms of the Proxy Advantage}\label{sec:results_anatomy}

Figure~\ref{fig:scaling} quantifies the scaling argument for CAP, and Figure~\ref{fig:fls_heldout} does so for UFL under the SOTA deployment. Across the small CAP shapes, the proxy's solve time is nearly \emph{flat} in the instance dimension while the oracle's grows by orders of magnitude with the master-bound $1000\times100$ CAP shape constituting the exception noted in Section~\ref{sec:results_cap}. Thus, the speedup is set by the cost of exact Benders, and the true gap stays in a $0.0$--$1.5\%$ band with no upward trend in size. Structurally, proxy separation consists of a fixed-cost forward pass followed by a closed-form completion. Its cost is set by the network rather than the instance, whereas the exact subproblem solve scales with the recourse. As instance size increases, the computational gap between amortized proxy inference and exact recourse optimization widens.

\begin{figure}[!t]
\centering
\includegraphics[width=0.85\textwidth]{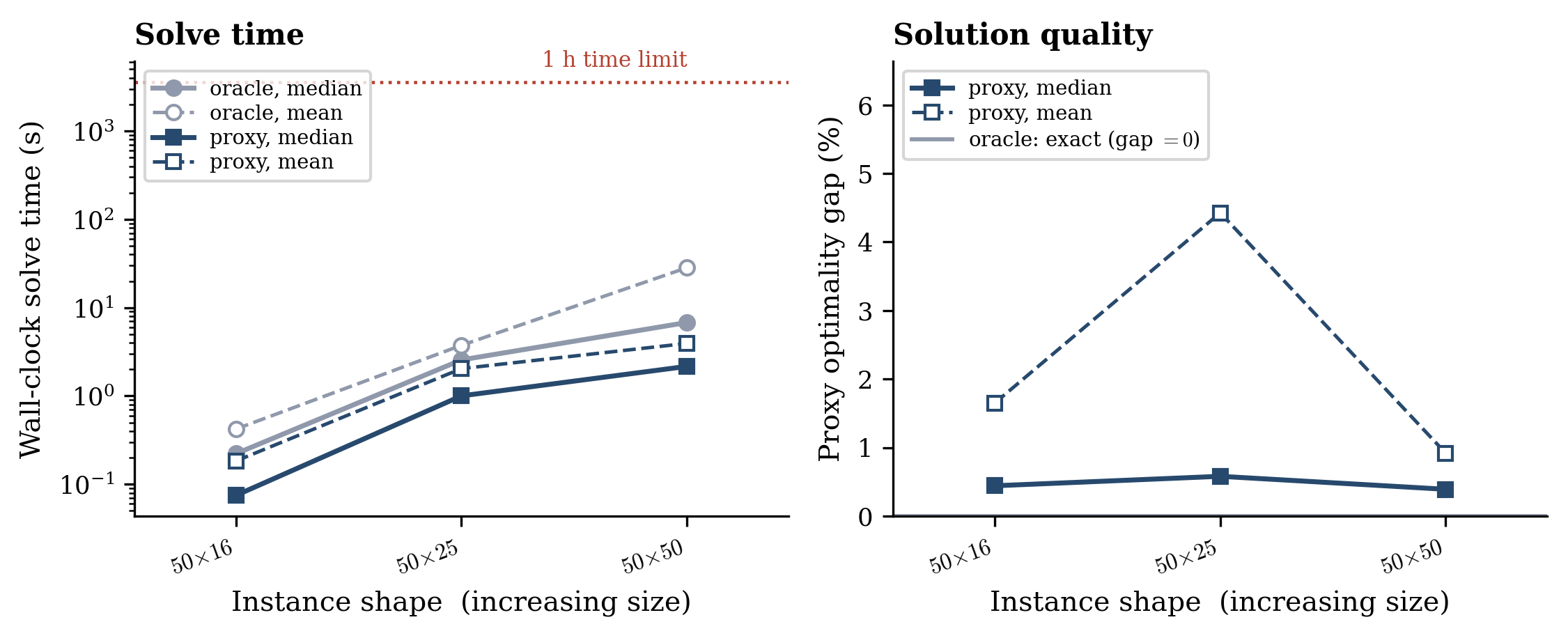}
\caption{Scaling of solve time and optimality gap with instance size for CAP: exact outer-loop Benders versus the proxy. The UFL counterpart under the SOTA deployment is Figure~\ref{fig:fls_heldout} (bottom).}
\label{fig:scaling}
\end{figure}

\noindent\textbf{Convergence behavior.} Figure~\ref{fig:convergence} contrasts the two solvers on a representative (median-speedup) CAP instance per shape, against the number of cuts separated. The exact solver drives its provable gap $1-\mathrm{LB}/\mathrm{UB}$ down one separation at a time, tracing the familiar decreasing curve and proving optimality costs in tens of cuts on the small CAP shapes. The proxy follows the same branch-and-cut logic, replacing each exact separation call with a fixed-cost forward pass. On every CAP shape, the marker sits close to the optimal axis at a small fraction of the oracle's cut count. The companion convergence figure for UFL is Figure~\ref{fig:fls_convergence} in Section~\ref{sec:results_fls}.

\begin{figure}[!t]
\centering
\includegraphics[width=0.85\textwidth]{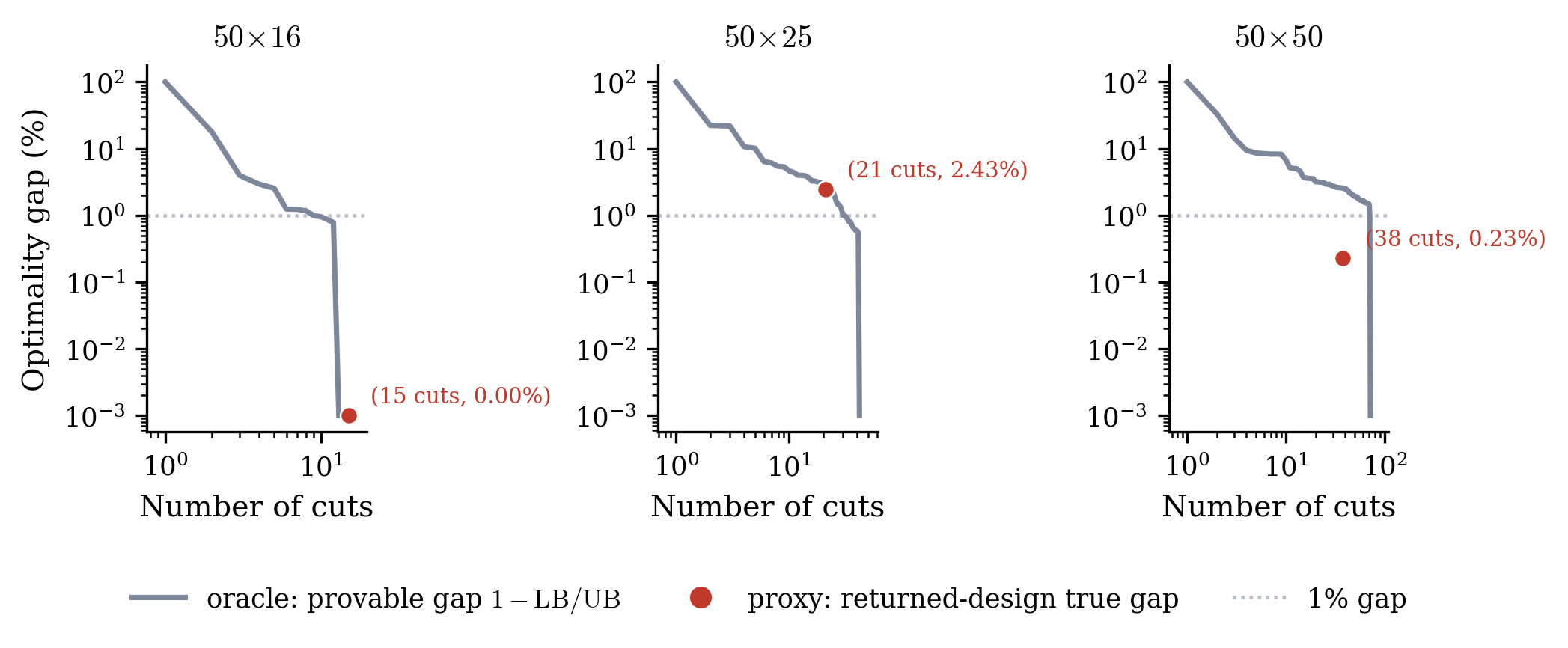}
\caption{Convergence of the optimality gap against the number of cuts separated for CAP. The exact solver traces its provable gap $1-\mathrm{LB}/\mathrm{UB}$; the proxy is shown as a single marker at the number of cuts it adds and the true gap of the design it returns. The UFL counterpart under the SOTA deployment is Figure~\ref{fig:fls_convergence}.}
\label{fig:convergence}
\end{figure}

\noindent\textbf{Validity and cut tightness.} Across the held-out runs, no proxy cut removed the true optimum or rendered the master infeasible. The quality of a proxy-only run therefore tracks how tightly the proxy's certified cuts approximate the exact ones. When the proxy-generated cuts closely approximate the oracle cuts, the resulting proxy fixed point lies within about $1\%$ of the optimum. When the cuts remain valid but loose, the resulting fixed point becomes suboptimal but never infeasible or invalid. The MCNDP UB comparison in Table~\ref{tab:mcndp_ub_comparison} instead compares cut-set deployment with the oracle incumbent: $10\times35$ is nearly tied, whereas $10\times60$ and $10\times83$ have larger mean UB gaps. The remaining limitation is the looseness of the MCNDP proxy's predicted duals at the larger shapes. Developing size-aware or more expressive MCNDP predictors, therefore, constitutes the natural next step for improving large-scale performance.

\subsection{Practical Insights}\label{sec:managerial}

Three observations summarize when and how Proxy-BD helps a practitioner considering it in their own decomposition pipeline.

\emph{Validity is structural.} Across the held-out evaluations, no proxy cut removed the true optimum or rendered the master infeasible, and the LP-free repair returned a demand-satisfying design on all 360 MCNDP instances. A poor proxy prediction, therefore, affects solution quality rather than validity or feasibility. The worst possible outcome is a suboptimal design returned with a valid bound. This separates Proxy-BD from heuristic learning-based decompositions whose outputs must be verified case by case.

\emph{Speedup follows the cost of the recourse.} The proxy's per-cut cost is a single forward pass plus a closed-form completion, near-constant in instance dimension, whereas an exact subproblem solve grows with the recourse. When the recourse dominates, e.g., the two largest UFL shapes and the two largest MCNDP shapes in this study, the proxy delivers a two- to three-order-of-magnitude acceleration. When the master search dominates, e.g., the $1000\times100$ CAP shape, the proxy is neutral or slower at the median because there is little subproblem cost to save. Practitioners should therefore expect the proxy to deliver speedup precisely on the instances where exact Benders is most expensive, and to be near-neutral elsewhere.

\emph{Deployment refinements address the right tail of the gap distribution.} When a single proxy cut is loose, e.g., under-trained large-scale UFL under the standard scheme (see Appendix \ref{C}) or MCNDP at $10\times60$ and $10\times83$ where the predicted duals under-price the coupled recourse, the deployment can compensate for cut looseness without retraining the underlying model. For UFL, aggregated cuts with true-cost selection (Section~\ref{sec:results_fls}) restore the failing shapes to sub-$0.7\%$ mean true gap. For MCNDP, the proxy-versus-oracle UB comparison in Table~\ref{tab:mcndp_ub_comparison} shows near-parity on $10\times35$ and growing UB deltas on $10\times60$ and $10\times83$, so improving predictor tightness is the main way to help the large shapes. The UFL refinement and MCNDP cut-set deployment preserve validity by construction. The MCNDP table is an incumbent-quality diagnostic rather than a cut-count trade-off study.

\section{Conclusions and Future Directions}\label{section:conclusions}

This paper introduced the Proxy-BD, in which the second-stage subproblem of a two-stage MIP is replaced by an optimization proxy that emits provably valid Benders cuts. The predict--project--complete certification layer of Section~\ref{sec:proxy_unified} converts any neural output into a dual-feasible certificate, so every emitted cut is a valid Benders cut by construction. Self-supervised training of the proxy (Section~\ref{sec:opt_block3}) maximizes the certified dual bound on oracle states without requiring optimal dual labels. For problems where the subproblem can be infeasible, a slice-normalized reformulation (Section~\ref{sec:neural_feasibility}) bounds the Farkas certificate set with a single linear constraint, after which the same predict--project--and--complete pipeline applies with exactly one extra bounded constraint in the projection step. An optional LP-free repair restores demand feasibility on the returned design when an exactly feasible solution must be returned. The held-out evaluation on capacitated facility location, uncapacitated facility location, and multicommodity capacitated network design (Section~\ref{sec:computational_results}) shows that the speedup over the exact oracle grows with the cost of the recourse, while the true optimality gap stays small wherever the proxy's cuts are tight. The B\&BC scheme of \citet{fischetti2017redesigning} serves as the UFL deployment, and the MCNDP cut-set deployment gives a scalable valid-cut baseline whose paired UB comparison shows small $10\times35$ deltas and larger $10\times60$ and $10\times83$ deltas.

\subsection{Summary of Contributions}

The contributions are summarized below.

\begin{enumerate}
    \item \textbf{Proxy Benders cuts with structural validity.} The predict--project--complete certification layer of    Section~\ref{sec:proxy_unified} converts any network output into a dual-feasible certificate; the resulting Benders cut is valid for every $y$ (Proposition~\ref{prop:proxy_cut_valid}), so training controls only the strength of the cut, not its validity.
    \item \textbf{Self-supervised training without dual labels.} Section~\ref{sec:opt_block3} formulates the training objective as the certified dual bound and supplies its completion subgradient, so the proxy learns from oracle states alone, with no exact-dual labels required.
    \item \textbf{Feasibility handling via slice-normalized Farkas certificates.} Section~\ref{sec:neural_feasibility} emits a separate feasibility cut from a slice-normalized certificate, with validity by linear programming duality and the slice constraint $\mathbf{e}^\top\lambda\le 1$ guaranteeing the optimal certificate is supported on an extreme ray of the Farkas cone (Lemma~\ref{lem:slice_extreme_rays}). An optional problem-specific LP-free repair restores demand feasibility on the returned design when needed.
    \item \textbf{Held-out evaluation across three problem families.} Section~\ref{sec:computational_results} reports the proxy's true optimality gap and speedup against fully configured exact oracles on CAP, UFL, and MCNDP under a uniform protocol, identifying the families and shapes on which the proxy accelerates Benders and the deployment refinements and diagnostics (cut aggregation with true-cost selection for UFL, cut-set-separated UB comparisons for MCNDP) that isolate where cut tightness controls incumbent quality.
\end{enumerate}

\subsection{Future Research Directions}

Several directions extend this work.

\begin{enumerate}
    \item \textbf{Approximation analysis.} Characterizing the approximation quality of proxy fixed points relative to exact Benders solutions, and understanding how bounded cut error propagates through the decomposition process, remain important theoretical directions for future work.
    \item \textbf{Size-invariant proxy architectures.} The current per-shape MCNDP models are undertrained at the larger shapes (Section~\ref{sec:results_mcndp}). Graph neural networks that share parameters across instance dimensions would amortize training and may close the gap on shapes the per-shape MLP cannot reach.
    \item \textbf{Stochastic and multi-stage decomposition.} The predict--project--complete pipeline extends directly to stochastic Benders, where a single proxy can serve many scenario subproblems, and to multi-stage problems with nested recourse, where the gain of replacing repeated LP solves is amplified.
    \item \textbf{Hybrid proxy--exact deployments.} Adaptive schemes that invoke exact separation selectively while relying on the proxy elsewhere could combine the scalability of amortized inference with occasional oracle refinement inside large-scale decomposition frameworks.
\end{enumerate}

\subsection{Closing Remarks}

The Proxy-BD shows that an optimization proxy can replace the subproblem solve in classical Benders without losing cut validity, and that the speedup so obtained grows with the cost of the recourse. The framework integrates with the various B\&BC variants \citep{fischetti2017redesigning,moreno2019branch,satici2026branch}, and admits deployment refinements that tighten cuts without retraining the model. The combination of structural validity guarantees and self-supervised training without labels is intended to serve as a basis for further work that integrates optimization proxies into classical operations research decomposition methods.

\section*{Acknowledgments}

This research was partly supported by the NSF AI Institute for Advances in Optimization (Award 2112533).

\section*{Code and Data Disclosure}\label{sec:Code and Data Disclosure}

The code and data supporting the numerical experiments will be made available upon publication.

\begin{spacing}{1}
\typeout{}
\bibliographystyle{apalike}
\bibliography{References.bib}
\end{spacing}    

\begin{appendices}
\renewcommand{\thesection}{\Alph{section}}

\section{Box Normalization}
\label{A}

\noindent\textbf{Setup.} Section~\ref{sec:neural_feasibility} resolves the scaling ambiguity of Farkas certificates by intersecting the cone $\mathcal R=\{\lambda\ge 0:\;B^\top\lambda\le 0\}$ with the slice $\mathbf e^\top\lambda\le 1$. A natural alternative is to bound each multiplier individually, i.e., replace the single linear functional by a box $0\le\lambda\le w$ with $w\in\mathbb R^m_{+}$. This section records the box normalization, proves cut validity, and explains, on the two-dimensional geometry of Example~\ref{ex:box-mixture}, why box certificates need not lie on extreme rays of $\mathcal R$ and therefore can yield strictly weaker cuts than the slice variant.

\begin{figure}[!t]
    \centering
    \begin{subfigure}[t]{0.48\textwidth}
        \centering
        \resizebox{\linewidth}{!}{%
        \begin{tikzpicture}[scale=1.1]
            \begin{scope}
                \clip (-2.5,-2.5) rectangle (1.3,1.3);
                \fill[gray!15] (-2.5,-2.5) rectangle (0,0);
                \draw[red,thick] (-2.5,2.5) -- (2.5,-2.5);
            \end{scope}
            \draw[->] (-2.5,0) -- (1.3,0) node[right] {$d_1$};
            \draw[->] (0,-2.5) -- (0,1.3) node[above] {$d_2$};
            \draw[thick] (-2.5,0) -- (0,0); 
            \draw[thick] (0,-2.5) -- (0,0); 
            \node[gray!70] at (-1.55,-1.55) {$\mathcal{K}$};
            \node[red] at (-1.65,1.05) {$d_1+d_2\le 0$};
            \node at (-1.2,0.25) {$d_2\le 0$};
            \node at (0.35,-1.2) {$d_1\le 0$};
            \fill (0,0) circle (1.5pt);
            \node[below left] at (0,0) {$0$};
        \end{tikzpicture}%
        }
        \caption{$d$-space (cut).}
    \end{subfigure}
    \hfill
    \begin{subfigure}[t]{0.48\textwidth}
        \centering
        \resizebox{\linewidth}{!}{%
        \begin{tikzpicture}[scale=2.8]
            \begin{scope}
                \clip (-0.15,-0.15) rectangle (1.35,1.35);
                \fill[green!10] (0,0) rectangle (1.35,1.35);
                \node[green!40!black] at (0.55,1.18) {$\mathcal{R}$};
                \fill[gray!20] (0,0) rectangle (1,1);
                \draw[gray!70,dashed,thick] (0,0) rectangle (1,1);
                \node[gray!70] at (0.55,0.55) {$\mathcal{R}(\mathbf{1})$};
            \end{scope}
            \draw[->] (-0.15,0) -- (1.35,0) node[right] {$\lambda_1$};
            \draw[->] (0,-0.15) -- (0,1.35) node[above] {$\lambda_2$};
            \draw[->] (0.18,0.09) -- (0.95,0.475) node[above right] {$d(\bar{y})=(1,\tfrac12)$};
            \fill[red] (1,1) circle (0.75pt);
            \node[red,anchor=west] at (1,1) {$\lambda^{\text{box}}=(1,1)$};
        \end{tikzpicture}%
        }
        \caption{$\lambda$-space (certificate).}
    \end{subfigure}
    \caption{Left: the recourse-feasible cone $\mathcal{K}$ in $d$-space and
    the induced (non-facet) cut $d_1+d_2\le 0$ obtained from $\lambda^{\text{box}}=(1,1)$. Right: the certificate
    cone $\mathcal{R}$ and the box-normalized feasible region $\mathcal{R}(\mathbf{1})$; maximizing
    $d(\bar{y})^\top \lambda$ yields $\lambda^{\text{box}}=(1,1)$, which maps to the red cut on the left.}
    \label{fig:box-mixture}
\end{figure}
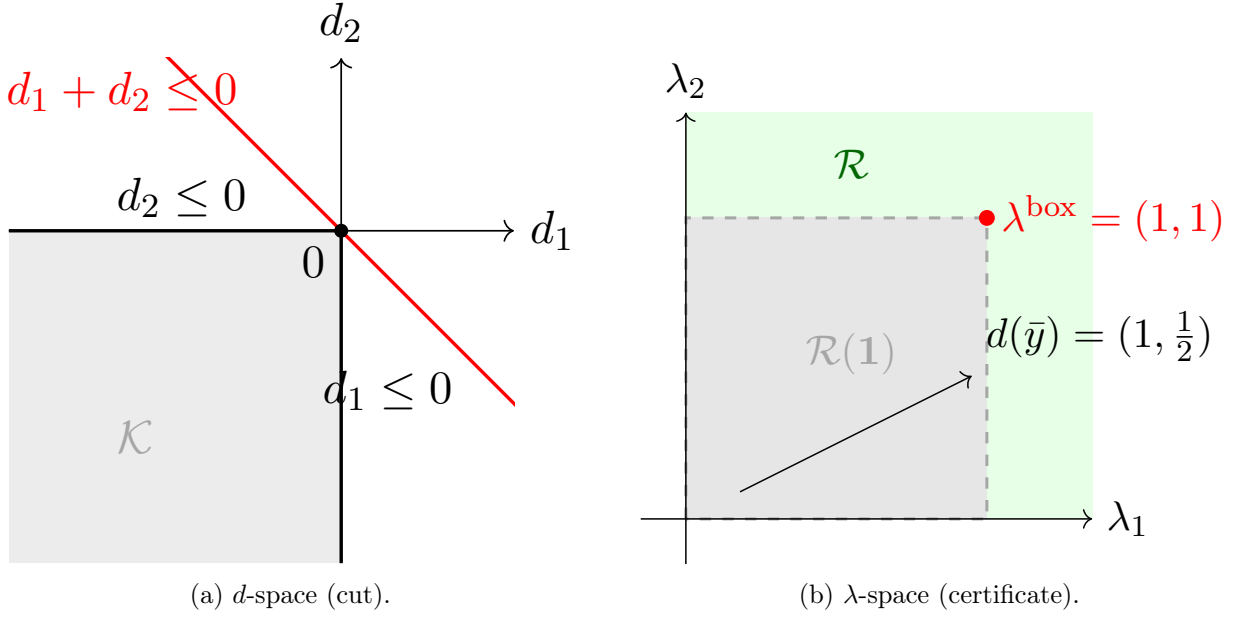

\noindent\textbf{Box Normalization.} Introduce row-wise slacks $s\ge 0$ and minimize a weighted violation of the recourse system~\eqref{eq:recourse_feasibility}:
\begin{equation}
\label{eq:phase1_box_primal_app}
v_{\mathrm{box}}(\bar y)\;:=\;\min_{x\ge 0,\;s\ge 0}\bigl\{\,w^\top s\;:\;Bx+s\ge d(\bar y)\,\bigr\}.
\end{equation}
By construction~\eqref{eq:phase1_box_primal_app} is always feasible, and $v_{\mathrm{box}}(\bar y)=0$ iff~\eqref{eq:recourse_feasibility} is feasible at $\bar y$. Its LP dual is
\begin{equation}
\label{eq:phase1_box_dual_app}
\max_{\lambda}\bigl\{\,d(\bar y)^\top\lambda\;:\;B^\top\lambda\le 0,\;0\le\lambda\le w\,\bigr\}.
\end{equation}
The box constraint makes the dual feasible region a nonempty compact polytope, so~\eqref{eq:phase1_box_dual_app} admits an optimizer $\lambda^\star\in\mathcal R(w):=\{\lambda\in\mathcal R:\lambda\le w\}\subseteq\mathcal R$.

\begin{prop}\label{prop:box_norm_valid}
Let $\lambda^\star$ be any optimal solution of~\eqref{eq:phase1_box_dual_app}. If $v_{\mathrm{box}}(\bar y)>0$, then~\eqref{eq:recourse_feasibility} is infeasible at $\bar y$, the inequality $d(y)^\top\lambda^\star\le 0$ is valid for every $y$ admitting feasible recourse, and it separates $\bar y$ with value $d(\bar y)^\top\lambda^\star=v_{\mathrm{box}}(\bar y)>0$.
\end{prop}
\begin{proof}
The primal~\eqref{eq:phase1_box_primal_app} is feasible and the dual~\eqref{eq:phase1_box_dual_app} has a nonempty compact feasible region, so strong duality yields $d(\bar y)^\top\lambda^\star=v_{\mathrm{box}}(\bar y)$. If $v_{\mathrm{box}}(\bar y)>0$, then~\eqref{eq:recourse_feasibility} is infeasible at $\bar y$ by the same primal--dual equivalence used in Theorem~\ref{thm:phase1_slice_farkas}. Since $\lambda^\star\in\mathcal R(w)\subseteq\mathcal R$, weak duality applied to~\eqref{eq:recourse_feasibility} at any recourse-feasible $y$ gives $d(y)^\top\lambda^\star\le 0$, exactly the structural form of the slice cut~\eqref{eq:feas_cut_lp} with the box-induced certificate.
\end{proof}

\noindent\textbf{Geometry on Example~\ref{ex:box-mixture}.} Take $m=2$, $B=-I_2$, $X=\mathbb R^2_+$, $w=\mathbf 1$, and $d(\bar y)=(1,\tfrac12)$ as in Example~\ref{ex:box-mixture}. Maximizing $d(\bar y)^\top\lambda=\lambda_1+\tfrac12\lambda_2$ over $\mathcal R(\mathbf 1)=[0,1]^2$ yields the unique optimizer $\lambda^{\mathrm{box}}=(1,1)$, sitting at the corner of the box rather than on an extreme ray of $\mathcal R=\mathbb R^2_+$. The induced feasibility cut is $d_1(y)+d_2(y)\le 0$, which is valid by Proposition~\ref{prop:box_norm_valid} but strictly dominated by the facet system $\{d_1\le 0,\;d_2\le 0\}$ of the recourse-feasible cone $\mathcal K=\{d:d_1\le 0,\,d_2\le 0\}$: the point $d=(1,-2)$ violates feasibility yet satisfies $d_1+d_2\le 0$. In contrast, the slice certificate $\lambda^{\mathrm{slice}}=(1,0)$ from Example~\ref{ex:box-mixture} hits an extreme ray of $\mathcal R$ via Lemma~\ref{lem:slice_extreme_rays} and recovers the facet cut $d_1(y)\le 0$. The box and slice certificates therefore deliver \emph{different} certified separation values, $v_{\mathrm{box}}(\bar y)=\lambda_1^{\mathrm{box}}+\tfrac12\lambda_2^{\mathrm{box}}=\tfrac32$ at $\lambda^{\mathrm{box}}=(1,1)$ versus $t^\star(\bar y)=1$ at $\lambda^{\mathrm{slice}}=(1,0)$; the larger box separation value does not translate into a stronger cut, but only the slice variant is guaranteed by Lemma~\ref{lem:slice_extreme_rays} to anchor the proxy~\eqref{eq:proxy_Lfeas} on a facet of $\mathcal K$. Figure~\ref{fig:box-mixture} illustrates the cut geometry in $d$-space.

\section{Explicit Predict--Project--Complete Derivations}
\label{B}

This appendix carries out the predict--project--complete pipeline of Section~\ref{sec:proxy_unified} explicitly on the dual block of each recourse subproblem, for the deployments evaluated in Section~\ref{sec:computational_results}: CAP, UFL under the plain B\&BC scheme, UFL under the accelerated scheme of \citet{fischetti2017redesigning}, and MCNDP with slice-normalized Farkas feasibility cuts. A tilde denotes the raw network output and a hat denotes the sign-cone projection. Cut validity holds for \emph{any} network output by Proposition~\ref{prop:proxy_cut_valid}, so each derivation needs only to exhibit a dual-feasible completion.

\subsection{CAP: Classic Benders Decomposition}
\label{A.1}

\noindent\textbf{Recourse subproblem.} For the CFLP primal in~\eqref{eq:cflp_primal} and fixed $\bar y\in Y$, write each constraint with its dual multiplier in a dedicated column:

{\allowdisplaybreaks
\begin{align*}
Q(\bar y)\;=\;\min_{x\ge 0}\quad
& \textstyle\sum_{i,j} d_i c_{ij}\,x_{ij} &&&& \\
\text{s.t.}\quad
& \textstyle\sum_j x_{ij}\ge 1 && \forall i && (\lambda_i\ge 0),\\
& \textstyle\sum_i d_i x_{ij}\le s_j\bar y_j && \forall j && (\mu_j\ge 0),\\
& x_{ij}\le \bar y_j && \forall i,j && (\nu_{ij}\ge 0).
\end{align*}    
}

Associating the listed multipliers with the rows and applying linear-programming duality, each primal column contributes one dual inequality: the variable $x_{ij}\ge 0$ appears in the demand row~$i$ (coefficient $1$), the capacity row~$j$ (coefficient $d_i$, a $\le$-row, so $\mu_j$ enters with a minus sign), and its box row $(i,j)$ (coefficient $1$, multiplier $\nu_{ij}$), giving $\lambda_i-d_i\mu_j-\nu_{ij}\le d_ic_{ij}$. The dual objective collects the right-hand sides of the three constraint families ($1$, $s_j\bar y_j$, and $\bar y_j$):
\begin{align}
Q(\bar y) = \max_{\lambda,\mu,\nu\ge 0}
\Bigl\{ \textstyle\sum_i \lambda_i - \sum_j s_j\bar y_j\mu_j - \sum_{i,j}\bar y_j\nu_{ij}
\ :\ \lambda_i - d_i\mu_j - \nu_{ij}\le d_i c_{ij}\ \forall i,j\Bigr\}.
\label{eq:cap_dual_app}
\end{align}

\noindent\textbf{Predict.} The proxy outputs $\tilde\lambda\in\mathbb{R}^m$, one component per client. The unconstrained network has no built-in sign guarantee, so dual feasibility of $\tilde\lambda$ is not assured.

\noindent\textbf{Project.} Set $\hat\lambda = (\tilde\lambda)_+$, the componentwise ReLU. This projects the prediction onto the nonnegative orthant $\lambda\ge 0$ that the demand constraints impose on the demand multipliers. Cut validity (Proposition~\ref{prop:proxy_cut_valid}) requires only \emph{dual feasibility}, so the projection step enforces the sign condition that the network output may otherwise violate; the remaining capacity and box inequalities $\lambda_i - d_i\mu_j - \nu_{ij}\le d_ic_{ij}$ are not enforced here but are deferred to the completion step.

\noindent\textbf{Complete.} For fixed $\hat\lambda$, the dual~\eqref{eq:cap_dual_app} decouples by facility~$j$ in the remaining $(\mu_j,\nu_{\cdot j})$ block:
\begin{equation*}
\min_{\mu_j,\nu_{\cdot j}\ge 0}\Bigl\{ s_j\mu_j + \textstyle\sum_i\nu_{ij}\ :\ d_i\mu_j+\nu_{ij}\ge \hat\lambda_i - d_i c_{ij}\ \forall i\Bigr\},
\end{equation*}
which is the LP dual of the continuous knapsack
\begin{equation*}
\kappa_j(\hat\lambda) := \max\Bigl\{\textstyle\sum_i(\hat\lambda_i - d_i c_{ij})\,a_i\ :\ 0\le a_i\le 1,\ \sum_i d_i a_i\le s_j\Bigr\}.
\end{equation*}
The continuous knapsack admits the standard greedy closed form: writing $r_{ij}:=\hat\lambda_i - d_i c_{ij}$ and densities $q_{ij}:=r_{ij}/d_i$, sort the items with $r_{ij}>0$ by decreasing $q_{ij}$ and fill capacity $s_j$ until exhausted, with the breaking item~$b$ taken fractionally. The completion variables are then read off from the knapsack KKT threshold density $q_{bj}$:
\begin{equation*}
\hat\mu_j^\star = q_{bj} = \tfrac{\hat\lambda_b - d_b c_{bj}}{d_b},
\qquad
\hat\nu_{ij}^\star = \bigl[\,r_{ij} - d_i\hat\mu_j^\star\,\bigr]_+ = d_i\bigl[\,q_{ij} - q_{bj}\,\bigr]_+ \quad \forall i.
\end{equation*}
Only customers with density strictly above $q_{bj}$ receive a positive $\hat\nu_{ij}^\star$. If every positive-profit item fits within $s_j$, the capacity is slack and one takes $\hat\mu_j^\star=0$ and $\hat\nu_{ij}^\star=[\,r_{ij}\,]_+$.

\noindent\textbf{Cut.} By strong duality of the per-facility completion, $\kappa_j(\hat\lambda)=s_j\hat\mu_j^\star+\sum_i\hat\nu_{ij}^\star$, so substituting $(\hat\mu^\star,\hat\nu^\star)$ into the $\bar y$-dependent terms of~\eqref{eq:cap_dual_app} gives the cut slope $\beta_j=-(s_j\hat\mu_j^\star+\sum_i\hat\nu_{ij}^\star)=-\kappa_j(\hat\lambda)$ and intercept $\alpha=\sum_i\hat\lambda_i$. The proxy optimality cut at any $y$ is therefore
\begin{equation*}
\theta\ \ge\ \textstyle\sum_i\hat\lambda_i - \sum_j \kappa_j(\hat\lambda)\,y_j,
\end{equation*}
which matches Example~\ref{sec:opt_cflp_example}.

\subsection{UFL: Classic B\&BC}
\label{A.2}

\noindent\textbf{Proxy architecture (row-wise).}
CAP uses a full-state representation because its capacity rows tie all clients together through the shared facility multiplier $\mu_j$. UFL has no such coupling: conditional on a fixed Benders state $\hat y$, the canonical oracle multiplier $\pi_i^\star(\hat y)$ for client $i$, i.e, the optimal assignment cost $\min_{j:\hat y_j=1}c_{ij}$, depends only on that client's cost row $c_{i,:}$ and on $\hat y$, not on the other rows and not on the opening costs $f$. The natural proxy is therefore a single network shared across clients and applied row by row,
\[
h_\theta(c_{i,:},\hat y)=\hat\pi_i ,
\]
The opening costs $f$ only appear in the master objective, where they determine which designs to explore. Once the decisions $\hat{y}$ are fixed, the subproblem and its certificate no longer depend on $f$. As a result, the proxy model never needs $f$ as an input. Additionally, because the certified bound is simply a sum across all clients, this speeds up training on large instances. Instead of evaluating every client, a sum over a random subset $I$ takes place before multiplying by $m/|I|$ to get an unbiased estimate of the full bound. During deployment, however, the generated cut evaluates all clients.

\noindent\textbf{Recourse subproblem.}
The $M$-instance UFL has no capacity constraint and nonnegative assignment costs $c_{ij}\ge 0$; a master side constraint $\sum_j y_j\ge 1$ keeps at least one facility open, so the recourse is feasible at every visited design. At fixed $\bar y$,
\begin{align*}
Q(\bar y)\;=\;\min_{x\ge 0}\quad
& \textstyle\sum_{i,j} c_{ij}\,x_{ij} &&&& \\
\text{s.t.}\quad
& \textstyle\sum_j x_{ij}\ge 1 && \forall i && (\pi_i\ge 0),\\
& x_{ij}\le \bar y_j && \forall i,j && (\nu_{ij}\ge 0).
\end{align*}
Associating the listed multipliers and applying linear-programming duality, the column of $x_{ij}\ge 0$ appears in the demand row~$i$ (coefficient $1$) and its box row $(i,j)$ (coefficient $1$, a $\le$-row, so $\nu_{ij}$ enters with a minus sign), giving the single inequality $\pi_i-\nu_{ij}\le c_{ij}$; the dual objective collects the right-hand sides $1$ and $\bar y_j$:
\begin{align}
Q(\bar y) = \max_{\pi,\nu\ge 0}\Bigl\{\textstyle\sum_i\pi_i - \sum_{i,j}\bar y_j\nu_{ij}\ :\ \pi_i - \nu_{ij}\le c_{ij}\ \forall i,j\Bigr\}.
\label{eq:ufl_dual_app}
\end{align}

\noindent\textbf{Predict.} The proxy outputs $\tilde\pi\in\mathbb{R}^m$, one component per client, unconstrained in sign.

\noindent\textbf{Project.} Set $\hat\pi=(\tilde\pi)_+$, projecting onto the nonnegative orthant $\pi\ge 0$ induced by the assignment-demand constraints. As in CAP, validity needs only dual feasibility; the linking inequalities $\pi_i-\nu_{ij}\le c_{ij}$ are deferred to completion.

\noindent\textbf{Complete.} With no facility-capacity row there is no shared multiplier and no coupling across clients, so the completion decouples into one trivial program per pair $(i,j)$, i.e., the smallest $\nu_{ij}\ge 0$ with $\nu_{ij}\ge\hat\pi_i-c_{ij}$, i.e.\ the linear program $\min_{\nu_{ij}\ge 0}\{\nu_{ij}:\ \nu_{ij}\ge\hat\pi_i-c_{ij}\}$, whose optimizer is
\[
\hat\nu_{ij}(\hat\pi)=(\hat\pi_i-c_{ij})_+ ,
\]
a single elementwise pass with no LP solve and no sort: the continuous-knapsack completion of CAP with the capacity row absent.

\noindent\textbf{Cut (disaggregated).} Under the plain B\&BC deployment the recourse decomposes by client, $Q(y)=\sum_i Q_i(y)$, where $Q_i(y)=\min\{\sum_j c_{ij}x_{ij}:\ \sum_j x_{ij}\ge 1,\ 0\le x_{ij}\le y_j\}$ is the per-client assignment LP (equal to $\min_{j:\,y_j=1}c_{ij}$ at integer $y$). The proxy emits one cut per client on a per-client epigraph variable $\theta_i$,
\[
\theta_i\ \ge\ \hat\pi_i-\sum_j(\hat\pi_i-c_{ij})_+\,y_j ,
\]
which lower-bounds $Q_i(y)$ for \emph{all} $y$ (Proposition~\ref{prop:proxy_cut_valid}); the master represents the total recourse by $\sum_i\theta_i$. It therefore grows by one cut per client at each separated state. Appendix~\ref{B} extends this derivation to the SOTA deployment, where the per-client certificates are aggregated into a single cut per integer incumbent.

\subsection{UFL: SOTA deployment}
\label{A.3}

The accelerated deployment of Section~\ref{sec:results_fls} solves the \emph{same} $M$-instance UFL recourse as Appendix~\ref{B}, but inside the single-tree B\&BC scheme of \citet{fischetti2017redesigning} and with the per-client certificates \emph{aggregated} into one cut per incumbent. For completeness, the full pipeline is written below; the recourse, predict, project, and complete steps coincide with Appendix~\ref{B}, and only the cut assembly and the deployment rule differ.

\noindent\textbf{Recourse subproblem.} At a fixed integer incumbent $\bar y$ (with $\sum_j y_j\ge 1$ and assignment costs $c_{ij}\ge 0$, as in Appendix~\ref{B}), the recourse and its dual are
\begin{align*}
Q(\bar y)=\min_{x\ge 0}\ \Bigl\{\textstyle\sum_{i,j} c_{ij}x_{ij}\ :\ \sum_j x_{ij}\ge 1\ (\pi_i\ge0)\ \forall i,\ \ x_{ij}\le\bar y_j\ (\nu_{ij}\ge0)\ \forall i,j\Bigr\},
\end{align*}
\begin{align*}
Q(\bar y)=\max_{\pi,\nu\ge 0}\Bigl\{\textstyle\sum_i\pi_i-\sum_{i,j}\bar y_j\nu_{ij}\ :\ \pi_i-\nu_{ij}\le c_{ij}\ \forall i,j\Bigr\},
\end{align*}
where the dual is obtained column-by-column exactly as in Appendix~\ref{B}, and the recourse decomposes by client, $Q(\bar y)=\sum_i Q_i(\bar y)$ with $Q_i(\bar y)=\min_{j:\bar y_j=1}c_{ij}$.

\noindent\textbf{Predict.} The row-wise network outputs one multiplier per client, $\tilde\pi_i=h_\theta(c_{i,:},\bar y)$, stacked into $\tilde\pi\in\mathbb{R}^m$.

\noindent\textbf{Project.} $\hat\pi=(\tilde\pi)_+$, the componentwise ReLU onto the nonnegative orthant $\pi\ge 0$. The SOTA proxy uses this projected prediction as is.

\noindent\textbf{Complete.} As in Appendix~\ref{B}, with no facility-capacity row, each linking multiplier is completed independently and in closed form,
\[
\hat\nu_{ij}(\hat\pi)=(\hat\pi_i-c_{ij})_+ ,
\]
a single elementwise pass with no LP solve and no sort.

\noindent\textbf{Cut (aggregated).} The plain deployment of Appendix~\ref{B} keeps a per-client epigraph $\theta_i$ and emits the $m$ cuts $\theta_i\ge\hat\pi_i-\sum_j(\hat\pi_i-c_{ij})_+y_j$. The accelerated deployment instead introduces a \emph{single} aggregate recourse epigraph $\theta$ (master objective $\min\sum_j f_j y_j+\theta$, with $\theta$ standing for the total recourse $\sum_i Q_i(y)$) and \emph{sums} the per-client certificates into one cut with coefficients
\[
\alpha=\textstyle\sum_i\hat\pi_i ,
\qquad
\beta_j=-\sum_i(\hat\pi_i-c_{ij})_+ ,
\qquad
\theta\ \ge\ \alpha+\sum_j\beta_j\,y_j .
\]
Each per-client summand $\hat\pi_i-\sum_j(\hat\pi_i-c_{ij})_+y_j$ lower-bounds $Q_i(y)$ for every $y$ (Appendix~\ref{B}); summing over clients gives $\alpha+\sum_j\beta_j y_j\le\sum_i Q_i(y)=Q(y)$, so the aggregate inequality is a valid optimality cut by Proposition~\ref{prop:proxy_cut_valid}. The master therefore grows by \emph{one} cut per integer incumbent rather than one per client, which is the source of the one to two orders of magnitude reduction in cut count in Table~\ref{tab:headline_fls}.

\noindent\textbf{Deployment (SOTA driver, true-cost selection).} The aggregate cut is separated only at integer incumbents within a single branch-and-bound tree, in the accelerated style of \citet{fischetti2017redesigning}, with no explicit cutting-plane warmup. The returned design is not the master's terminal incumbent but the visited integer design of least \emph{exact total} cost $\sum_j f_j y_j+Q(y)$, with $Q(y)=\sum_i\min_{j:\,y_j=1}c_{ij}$ evaluated in closed form; this ranks only designs already visited during the run and requires no additional oracle calls.

\subsection{MCNDP: Classic B\&BC}
\label{A.4}

\noindent\textbf{Recourse subproblem.}
The network is a directed graph with node set $V$ and arc set $E$. Let $I$ be the commodity set; commodity $i$ has origin $o_i$, destination $t_i$, and demand $d_i$. Arc $j$ has tail $\mathrm{tail}(j)$, head $\mathrm{head}(j)$, capacity $s_j$, variable unit cost $c_j$, and fixed cost $f_j$. The variable $x_{ij}\in[0,\bar y_j]$ is the fraction of commodity $i$ routed on arc $j$, and $\delta^+(v),\delta^-(v)$ denote the arcs leaving and entering node $v$. Write $b_{iv}$ for the net supply of commodity $i$ at node $v$: $b_{i,o_i}=1$, $b_{i,t_i}=-1$, and $b_{iv}=0$ otherwise. For fixed $\bar y$ the recourse is
\begin{align*}
Q(\bar y)\;=\;\min_{x\ge 0}\quad
& \textstyle\sum_{i,j} d_i c_j\,x_{ij} \\
\text{s.t.}\quad
& \textstyle\sum_{j\in\delta^+(v)}x_{ij}-\sum_{j\in\delta^-(v)}x_{ij}\ \ge\ b_{iv} && \forall i,\ v && (u_{iv}\ge 0),\\
& \textstyle\sum_i d_i x_{ij}\le s_j\bar y_j && \forall j && (\rho_j\ge 0),\\
& x_{ij}\le \bar y_j && \forall i,j && (\tau_{ij}\ge 0).
\end{align*}
The first line is flow conservation for commodity $i$, written as outgoing minus incoming with a ``$\ge$'' sense (so the duals satisfy $u_{iv}\ge 0$): one unit must leave the origin, one unit must reach the destination, and every other node is balanced. The last two lines are the shared arc capacity and the per-arc open bound. The fixed design cost $\sum_j f_j y_j$ is not part of $Q(\bar y)$; it is folded into the cut below. The recourse is feasible at $\bar y$ if and only if the open arcs can route every commodity within their capacities; otherwise, the master separates a feasibility cut.

Associating the listed multipliers and applying linear-programming duality, the routing variable $x_{ij}\ge 0$ enters the tail flow row $\mathrm{tail}(j)$ (coefficient $+1$), the head flow row $\mathrm{head}(j)$ (coefficient $-1$), the shared capacity row $j$ (coefficient $d_i$, a $\le$-row, so $\rho_j$ enters with a minus sign), and its box row $(i,j)$ (coefficient $1$, multiplier $\tau_{ij}$), giving the dual feasibility row $u_{i,\mathrm{tail}(j)}-u_{i,\mathrm{head}(j)}-d_i\rho_j-\tau_{ij}\le d_ic_j$. The dual objective collects the constraint right-hand sides: $\sum_{i,v}b_{iv}u_{iv}$ from the flow rows and $-\sum_j s_j\bar y_j\rho_j-\sum_{i,j}\bar y_j\tau_{ij}$ from the capacity and box rows.

\noindent \textbf{Two dual programs: optimality cone and slice-Farkas certificate.} Two related dual programs share the dual cone of the recourse. The first is the standard recourse dual, used to derive optimality cuts at feasible $\bar y$:
\begin{align}
Q(\bar y)\;=\;
\max_{u,\rho,\tau\ge 0}
\Bigl\{
\textstyle\sum_{i,v} b_{iv}u_{iv}-\sum_j s_j\bar y_j\rho_j-\sum_{i,j}\bar y_j\tau_{ij}
\ :\
u_{i,\mathrm{tail}(j)}-u_{i,\mathrm{head}(j)}-d_i\rho_j-\tau_{ij}\le d_ic_j\ \forall i,j
\Bigr\}.
\label{eq:mcndp_dual_opt_app}
\end{align}
The second is the slice-normalized (Farkas) LP, used to derive feasibility cuts at infeasible $\bar y$. It optimizes over the same recourse-feasibility recession cone as~\eqref{eq:phase1_lp_dual}, but with a \emph{problem-specific} bounded normalization rather than the single functional $\mathbf e^\top\lambda\le 1$ of~\eqref{eq:phase1_slice_dual}: the recourse-feasibility recession cone (the cost-side row drops out, replaced by the homogeneous bound $\le 0$), with a per-commodity normalization $\sum_v u_{i,v}\le 1$ (one per commodity $i$) that bounds the Farkas certificate set:
\begin{align}
\max_{u,\rho,\tau\ge 0}
\Bigl\{
\textstyle\sum_{i} (u_{i,o_i}-u_{i,t_i})-\sum_j s_j\bar y_j\rho_j-\sum_{i,j}\bar y_j\tau_{ij}
\ :\
\ & u_{i,\mathrm{tail}(j)}-u_{i,\mathrm{head}(j)}-d_i\rho_j-\tau_{ij}\le 0\ \forall i,j,\nonumber\\[-1pt]
& \textstyle\sum_{v} u_{iv}\le 1\ \forall i
\Bigr\}.
\label{eq:mcndp_dual_farkas_app}
\end{align}
Cut validity follows solely from $(u,\rho,\tau)$ lying in this homogeneous recession cone, independently of the choice of normalization. When $\bar y$ is infeasible for the recourse, the LP~\eqref{eq:mcndp_dual_farkas_app} has a strictly positive optimum, and the corresponding $(u^\star,\rho^\star,\tau^\star)$ furnishes a separating feasibility cut $\alpha+\beta^\top y\le 0$ valid at every $y$ admitting feasible recourse. The per-commodity cap $\sum_v u_{iv}\le 1$ is not a verbatim instance of the single-functional slice $\mathbf e^\top\lambda\le 1$ of Theorem~\ref{thm:phase1_slice_farkas}, but exact detection still loses nothing: any positive Farkas ray $(u,\rho,\tau)$ can be rescaled by the single scalar $\gamma=1/\max_i\sum_v u_{iv}>0$, which preserves the homogeneous arc inequalities (so the shared $\rho_j$ scales consistently), meets every cap $\gamma\sum_v u_{iv}\le 1$, and keeps the objective strictly positive; hence infeasible recourse still yields a strictly positive optimum. This uses the cap as an \emph{inequality} ($\le 1$): the equality simplex $\sum_v u_{iv}=1$ would break the argument when a commodity has zero row-sum. The one-sided qualification therefore attaches only to a \emph{proxy-predicted} certificate, which may be too weak to separate (Proposition~\ref{prop:proxy_one_sided}).

\noindent\textbf{Predict.} The dual-tower proxy outputs two flow-conservation predictions, $\tilde u^{\text{opt}}, \tilde u^{\text{feas}}\in\mathbb{R}^{|I|\cdot|V|}$, one per cut family. Both heads share the same encoder of the Benders state $\xi=(\text{instance features}, \bar y)$ and differ only in their final linear layers.

\noindent\textbf{Project.} For the optimality head the projection is elementwise ReLU, $\hat u^{\text{opt}}_{iv}=(\tilde u^{\text{opt}}_{iv})_+$. For the feasibility head the projection enforces nonnegativity \emph{and} the per-slice constraint $\sum_v \hat u^{\text{feas}}_{iv}\le 1$: first the elementwise ReLU on $\tilde u^{\text{feas}}$, then the closed-form Euclidean projection of each row $(\hat u^{\text{feas}}_{i,v})_v$ onto the capped simplex $\{u\ge 0:\sum_v u_{iv}\le 1\}$ (the ``$\le 1$'' slice in~\eqref{eq:mcndp_dual_farkas_app}, not the equality simplex $\sum_v u_{iv}=1$). Each projection enforces only the sign (optimality head) or sign-and-slice (feasibility head) restrictions on the predicted $u$-block; it does \emph{not} by itself place $(\hat u,\rho,\tau)$ in the recourse dual cone, since the arc inequalities $u_{i,\mathrm{tail}(j)}-u_{i,\mathrm{head}(j)}-d_i\rho_j-\tau_{ij}\le d_ic_j$ (optimality) or $\le 0$ (feasibility) are enforced only after the $(\rho,\tau)$ completion below. Once completed, the full certificate is dual-feasible, and the emitted cut is valid by Proposition~\ref{prop:proxy_cut_valid} (optimality head) and Theorem~\ref{thm:phase1_slice_farkas} (feasibility head), regardless of network parameters.

\noindent\textbf{Complete.} The per-arc $(\rho_j,\tau_{\cdot j})$ completion has the same closed form for both heads. For fixed $\hat u$, the inner block of the dual decouples by arc; writing the per-commodity profit
\[
p_{ij}(\hat u)\;=\;\hat u_{i,\mathrm{tail}(j)}-\hat u_{i,\mathrm{head}(j)}\;-\;d_i c_j\ \ \text{(optimality head)},
\quad\text{or}\quad
p_{ij}(\hat u)\;=\;\hat u_{i,\mathrm{tail}(j)}-\hat u_{i,\mathrm{head}(j)}\ \ \text{(feasibility head)},
\]
the completion solves
\[
\min_{\rho_j,\tau_{\cdot j}\ge 0}\Bigl\{ s_j\rho_j+\textstyle\sum_i\tau_{ij}\ :\ d_i\rho_j+\tau_{ij}\ge p_{ij}(\hat u)\ \forall i\Bigr\},
\]
the LP dual of the continuous knapsack
\[
\kappa_j(\hat u)\;:=\;\max\Bigl\{\textstyle\sum_i p_{ij}(\hat u)\,a_i:\ 0\le a_i\le 1,\ \sum_i d_i a_i\le s_j\Bigr\}.
\]
This is the standard continuous knapsack with profit $p_{ij}$, weight $d_i$, and capacity $s_j$, solved in closed form: writing densities $q_{ij}:=p_{ij}(\hat u)/d_i$, discard the items with $p_{ij}\le 0$, sort the rest by decreasing $q_{ij}$, and fill $s_j$ until exhausted, with the breaking item~$b$ taken fractionally. The completion multipliers are then read off from the knapsack KKT threshold density $q_{bj}$,
\[
\hat\rho_j^\star=q_{bj}=\tfrac{p_{bj}(\hat u)}{d_b},
\qquad
\hat\tau_{ij}^\star=\bigl[\,p_{ij}(\hat u)-d_i\hat\rho_j^\star\,\bigr]_+=d_i\bigl[\,q_{ij}-q_{bj}\,\bigr]_+\ \ \forall i,
\]
exactly as in the CAP construction. By strong duality $\kappa_j(\hat u)=s_j\hat\rho_j^\star+\sum_i\hat\tau_{ij}^\star$.

\noindent\textbf{Cut.}
For the optimality head, substituting $(\hat\rho^\star,\hat\tau^\star)$ into the $\bar y$-dependent terms of~\eqref{eq:mcndp_dual_opt_app} gives slope $\beta_j=-\kappa_j(\hat u^{\text{opt}})$ and intercept $\hat\alpha=\sum_{i,v}b_{iv}\hat u^{\text{opt}}_{iv}$, i.e.\ $\theta\ge\hat\alpha-\sum_j\kappa_j(\hat u^{\text{opt}})\,y_j$. Folding the design cost through the total-objective epigraph $\eta=\theta+\sum_j f_jy_j$ yields the knapsack-lifted optimality cut
\[
\eta\ \ge\ \hat\alpha+\sum_j\bigl(f_j-\kappa_j(\hat u^{\text{opt}})\bigr)y_j .
\]
For the feasibility head, substituting $(\hat\rho^\star,\hat\tau^\star)$ into~\eqref{eq:mcndp_dual_farkas_app} gives slope $\beta_j=-\kappa_j(\hat u^{\text{feas}})$ and intercept $\alpha(\hat u^{\text{feas}})=\sum_{i}(\hat u^{\text{feas}}_{i,o_i}-\hat u^{\text{feas}}_{i,t_i})$, i.e.\ the slice-Farkas feasibility cut
\[
\alpha(\hat u^{\text{feas}})\;+\;\sum_j\bigl(-\kappa_j(\hat u^{\text{feas}})\bigr)\,y_j\;\le\;0,
\]
valid by Theorem~\ref{thm:phase1_slice_farkas} for every $\hat u^{\text{feas}}$ in the slice cone. At deployment, both heads are queried at each integer incumbent, and only the violated cut (if any) is added to the master.

\section{UFL Instances under Classic B\&BC}
\label{C}

Under classical B\&BC deployment, UFL shows the scaling effect over two orders of magnitude (see Table \ref{tab:ufl_classic}). On the four shapes from $100\times100$ to $500\times500$ the proxy holds a median true gap of $0.86$--$1.48\%$ while the median speedup climbs from $36\times$ to $672\times$. The exact oracle reaches the one-hour limit on much of $300\times300$ and $500\times500$, whereas the proxy returns its design in $\le 0.18\,\mathrm{s}$ up to $300\times300$ and in $6.94\,\mathrm{s}$ at $500\times500$. The two largest shapes, $1000\times1000$ and $2000\times2000$, are a boundary case: the proxy returns a design in under a second but at $70$--$71\%$ true gap. By the validity guarantee, these cuts remain correct, so the degradation is one of cut \emph{strength}: the per-shape model is undertrained at these sizes and emits cuts too loose to drive the master, and the search halts at a poor proxy fixed point. This is a limitation of the trained \emph{model}, not the cut machinery.

\begin{table}[!t]
\centering
\small
\caption{Held-out proxy vs.\ exact oracle. \emph{Gap} is the true optimality gap; \emph{Time} the median solve time; \emph{Speedup} the oracle/proxy ratio; $N$ paired instances. Bold marks entries where the proxy improves on the oracle ($<1.5\%$ gap, $>1\times$ speedup).}
\label{tab:ufl_classic}
\begin{tabular}{lrrrrrrr}
\toprule
 & & \multicolumn{2}{c}{Gap (\%)} & \multicolumn{2}{c}{Time (s)} & \multicolumn{2}{c}{Speedup} \\
\cmidrule(lr){3-4}\cmidrule(lr){5-6}\cmidrule(lr){7-8}
Shape & $N$ & median & mean & oracle & proxy & median & mean \\
\midrule
\quad $200\times200$   & 100 & \textbf{1.41} & 1.44 & 153.6            & 0.18 & $\mathbf{646\times}$ & $\mathbf{737\times}$ \\
\quad $300\times300$   & 100 & \textbf{1.48} & 1.34 & $90.6^{\dagger}$ & 0.18 & $\mathbf{672\times}$ & $\mathbf{1776\times}$ \\
\quad $500\times500$   & 100 & \textbf{0.86} & 0.91 & $3600^{\dagger}$ & 6.94 & $\mathbf{505\times}$ & $\mathbf{663\times}$ \\
\quad $1000\times1000$ &  20 & 71.4 & 71.3 & $3600^{\dagger}$ & 0.05 & --- & --- \\
\quad $2000\times2000$ &  20 & 70.1 & 70.1 & $3600^{\dagger}$ & 0.20 & --- & --- \\
\bottomrule
\end{tabular}

\vspace{3pt}
{\footnotesize $^{\dagger}$\, Oracle hit the one-hour limit on some or all instances of the shape; the gap is then against its best incumbent and the speedup a lower bound.}
\end{table}

\end{appendices}

\end{document}